\documentclass{amsart}

\usepackage{amssymb,latexsym,amscd}
\usepackage{hyperref}

\numberwithin{equation}{section}
\renewcommand{\thesection}{\arabic{section}}

\long\def\eatit#1{}

\newtheorem{thm}{Theorem}[subsection]
\newtheorem{prop}[thm]{Proposition}
\newtheorem{lem}[thm]{Lemma}
\newtheorem{cor}[thm]{Corollary}
\newtheorem{Ques}[thm]{Question}
\newtheorem{Prob}[thm]{Problem}
\newtheorem{Conj}[thm]{Conjecture}

\theoremstyle{definition}
\newtheorem{Def}[thm]{Definition}
\newtheorem{Not}[thm]{Notation}
\newtheorem{Eg}[thm]{Example}
\newtheorem{Ex}[thm]{Exercise}
\newtheorem*{Soln}{Solution}
\newtheorem{Rmk}[thm]{Remark}
\newtheorem*{Rmk*}{Remark}

\newcommand{\Spec}{\operatorname{Spec}} 
\newcommand{\Cox}{\operatorname{Cox}} 
\newcommand{\Num}{\operatorname{Num}} 
 
\newcommand{\shf}[1]{{\mathcal{#1}}} 

\newcommand{\pr}[1]{{{\bf P}^{#1}}}
\newcommand{\OO}{{\mathcal{O}}}

\newcommand{\EFF}{\operatorname{EFF}}      
\newcommand{\Cl}{\operatorname{Cl}}      
\newcommand{\NEF}{\operatorname{NEF}}

\begin{document}

\title[Global aspects of the geometry of surfaces]{Global aspects of the geometry of surfaces}

\author[B. Harbourne]{Brian Harbourne}
\address{Department of Mathematics\\
University of Nebraska\\
Lincoln, NE 68588-0130 USA}
\email{bharbour@math.unl.edu}

\date{July 30, 2009}

\thanks{Acknowledgments: These notes were prepared for lectures given at
the summer school of the SFB/TR 45 Bonn-Essen-Mainz, financed by 
the Deutsche Forschungsgemeinschaft, that took place March 23-27, 2009 
at the Pedagogical University of Cracow, with the goal of improving the training 
of PhD students and postdocs in the area, in particular of the members of the SFB/TR 45. 
I thank the organizers of the school, Stefan M\"uller-Stach and Tomasz Szemberg, 
for their invitation to give these lectures (and Tomasz for his careful reading of these notes), and 
I thank Joaquim Ro\'e,  fellow speaker at the Summer School in Krakow and my 
host in Barcelona the week before the School when some of the work on these notes was carried out.
I also thank Zach Teitler for sharing his notes \cite{refTe} on multiplier ideals with the participants,
and I thank Burt Totaro and J\'anos Koll\'ar for their comments regarding the Bounded Negativity Conjecture.}

\begin{abstract}
Several open problems related to
the behavior of the monoid of effective divisors and the nef cone
for smooth projective surfaces over an algebraically closed field are discussed,
motivating and putting into historical context concepts such as Mori dream spaces, Seshadri constants
and the resurgence of homogeneous ideals in polynomial rings.
Some recent work on these topics is discussed along with
the problem of which ordinary powers of homogeneous ideals contain given
symbolic powers of those ideals. Exercises, with solutions, are included.
\end{abstract}

\maketitle

\renewcommand{\thesection}{\Roman{section}}

\section{Lecture: Bounded Negativity}

\subsection{Introduction}\label{intro}

A {\em surface} here will always mean a smooth irreducible closed 2 dimensional 
subscheme of projective space, over an algebraically closed field $k$.
A {\em prime divisor} on $X$ is a reduced irreducible curve.

\begin{Not} Let $X$ be a surface. The divisor class group $\Cl(X)$ is the free abelian
group $G$ on the prime divisors, modulo linear equivalence. By intersection theory we have
a bilinear form on $G$ which descends to $\Cl(X)$. Two divisors which induce the same intersections
on curves (which in our situation are themselves divisors, since $X$ is a surface)
are said to be {\it numerically equivalent}.
We will denote numerical equivalence on divisors by $\sim$; thus $D_1\sim D_2$ means that
$D_1\cdot C=D_2\cdot C$ for all curves $C$. We denote $\Cl(X)$ modulo
numerical equivalence by $\Num(X)$. 
Also, $\EFF(X)$ denotes the submonoid of $\Cl(X)$
consisting of the classes of effective divisors on $X$, and $\NEF(X)$ denotes the
submonoid of $\Cl(X)$ of all classes $F$ such that $F\cdot C\ge0$ for all
$C\in\EFF(X)$.
\end{Not}

A basic fact is that $\Num(X)$ is a free abelian group of finite rank.
If $X$ is rational, then $\Num(X)=\Cl(X)$. If $X$ is obtained by blowing up points 
$p_1,\ldots,p_r\in\pr2$, for example, then $\Cl(X)$ is the
free abelian group on the class $L$ of the pullback of a line and on the classes
$E_i$ of the blowings up of the points $p_i$, hence $\Cl(X)$ has rank $r+1$.
The intersection form on $\Cl(X)$ in this case is defined by
$-L^2=E_i^2$ with $L\cdot E_i=E_j\cdot E_i=0$ for all $i$ and all $j\ne i$, and the canonical class is
$K_X=-3L+E_1+\cdots +E_r$.

We now recall the Adjunction Theorem and Riemann-Roch for surfaces:

\begin{thm}[Adjunction] Let $C$ be a prime divisor on a surface $X$ and let $K_X$ 
be the canonical class on $X$. 
Then there is a non-negative integer $p_C$ such that $C^2=2p_C-2-C\cdot K_X$.
\end{thm}

\begin{thm}[Riemann-Roch]\label{RR} Given any divisor $D$ on a surface $X$,
let $\chi(\OO_X(D))$ denote $h^0(X,\OO_X(D))-h^1(X,\OO_X(D))+h^2(X,\OO_X(D))$.
Then 
$$\chi(\OO_X(D))=\frac{D^2-K_X\cdot D}{2}+\chi(\OO_X).$$
\end{thm}

Riemann-Roch becomes especially useful when taken together with Serre duality,
which for a surface $X$ says that $h^i(X,\OO_X(D))=h^{2-i}(X,\OO_X(K_X-D))$.
Castelnuovo's criterion for rationality is also useful: 

\begin{thm}[Castelnuovo]\label{CCR} A surface $X$ is rational if and only if
$h^0(X,\OO_X(2K_X))=h^1(X,\OO_X)=0$.
\end{thm}

Next, we recall the Hodge Index Theorem:

\begin{thm}[HIT]\label{HIT} Given a surface $X$ and $D\in\Num(X)$ with $D^2>0$,
then the intersection form on the space $D^{\perp}\subseteq \Num(X)$ of classes $F$ with $F\cdot D=0$
is negative definite.
\end{thm}

Finally, we recall the semicontinuity principle. We say points $p_1,\ldots,p_r$ are {\it essentially distinct\/} points of 
a surface $X$, and that $X_{r+1}$ is the blow up of $X$ at $p_1,\ldots,p_r$, 
if $p_1\in X=X_1$, $\pi_1:X_2\to X_1$ is the blow up of $X_1$ at $p_1$, and
for $1<i\le r$ we have $p_i\in X_i$, and $\pi_i:X_{i+1}\to X_i$ is the blow up of $X_i$ at $p_i$.
By identifying $X_{i+1}$ with $X_i$ away from $p_i$, we can regard $p_{i+1}$ as being in $X_i$ when
$\pi_i(p_{i+1})\ne p_i$. In this way distinct points $p_1,\ldots,p_r\in X$ can be regarded as being essentially
distinct. Let $\pi_{j,i}:X_j\to X_i$ be the morphism $\pi_{j-1}\circ\cdots\circ\pi_i$ whenever $j>i$,
and let $E_i$ be the divisor (or divisor class, depending on context) given by $\pi_{r,i}^{-1}(p_i)$.
When $X=\pr2$, let $L$ be the pullback to $X_{r+1}$ of a general line on $X$.
Otherwise we will assume $L$ is the pullback to $X_{r+1}$ of some ample divisor on $X$.

\begin{thm}[Semicontinuity Principle]\label{SCprin} Let $a,a_1,\ldots,a_r$ be
integers, let $p_1,\ldots,p_r$ be general points of $X=\pr2$ and denote by $X_{r+1}$ the blow 
up of $X$ at $p_1,\ldots,p_r$, with $L,E_1,\ldots,E_r$ being the usual associated classes.
Also, given essentially distinct points $p_1',\ldots,p_r'$ of $X$, let $X'_{r+1}$ denote the blow 
up of $X$ at $p_1',\ldots,p_r'$, and let $L',E_1',\ldots,E_r'$ denote the associated classes.
\begin{itemize}
\item[(a)] If $aL-\sum_ia_iE_i\in\EFF(X_{r+1})$, 
then $aL'-\sum_ia_iE_i'\in\EFF(X'_{r+1})$ for every choice of essentially distinct points 
$p_1',\ldots,p_r'$ of $X$. 
\item[(b)] If $aL'-\sum_ia_iE_i'\in\NEF(X'_{r+1})$ for some choice of essentially distinct points 
$p_1',\ldots,p_r'$ of $X$ and if $(aL'-\sum_ia_iE_i')^2>0$, 
then $aL-\sum_ia_iE_i\in\NEF(X_{r+1})$.
\end{itemize}
\end{thm}

\begin{proof} Following \cite{refK}, parameterize essentially distinct points of $X$ by
schemes $W_i$ where we set $b_0:W_1\to W_0$ to be $X\to\Spec{(k)}$
and recursively we define $W_{i+1}\to W_i\times_{W_{i-1}}W_i$ to be the blow up of the diagonal
in $W_i\times_{W_{i-1}}W_i$, setting $b_i:W_{i+1}\to W_i$ to be the composition of
$W_{i+1}\to W_i\times_{W_{i-1}}W_i$ with the projection $\pi_{1i}: W_i\times_{W_{i-1}}W_i\to W_i$ to the first factor. The morphisms $b_i$ are smooth (\cite[17.3, 19.4]{refEGA}). [Here is a proof.
Note $b_0$ is smooth (since $X$ is). Assuming $b_{i-1}$ is smooth, we see the projection
$W_i\times_{W_{i-1}}W_i\to W_i$ is smooth (\cite[Proposition III.10.1(d)]{refHr}), and the exceptional locus
$B_{i+1}\subset W_{i+1}$ for $b_i$ is smooth (\cite[Theorem II.8.24(b)]{refHr}) and locally isomorphic to
$W_i\times\pr1$. We now see that $b_i$ is smooth by checking surjectivity of the induced maps on Zariski
tangent spaces (\cite[Proposition III.10.4(iii)]{refHr}) at points $x\in W_{i+1}$. Away from $B_{i+1}$,
$W_{i+1}\to W_i\times_{W_{i-1}}W_i$ is an isomorphism and $W_i\times_{W_{i-1}}W_i\to W_i$ 
is smooth, hence surjectivity follows for points $x\not\in B_{i+1}$. At points $x\in B_{i+1}$, 
the composition $B_{i+1}\subset W_{i+1}\to W_i$ is smooth and thus the map
on tangent spaces induced by $B_{i+1}\to W_i$ is already surjective at $x$, hence so is the one induced by
$W_{i+1}\to W_i$. Thus $b_i$ is smooth.]

Consider the pullbacks $B'_i$ to $W_{r+1}$ of the divisors $B_i$. For any ample divisor $L$ on $X$, let $L''$
be the pullback to $W_{r+1}$ via the blow ups $b_i$ and the projections $\pi_{2i}$ 
on the second factors. Let $\shf{F}=\OO_{W_{r+1}}(aL''-\sum_ia_iB'_i)$.
Then for any essentially distinct points $p_1',\ldots,p_r'$ of $X$ we have a uniquely determined
point $w\in W_r$, the fiber $(W_{r+1})_w$ of $W_{r+1}$ over $w$ is $X'_{r+1}$,
and the restriction $\shf{F}_w$ of $\shf{F}$ to $(W_{r+1})_w$ is $\OO_{X'_{r+1}}(aL'-\sum_ia_iE_i')$.
By the semicontinuity theorem (\cite[Theorem III.12.8]{refHr}), 
$h^0((W_{r+1})_w, \shf{F}_w)$ is an upper semicontinuous function of $w$. This implies (a).

Now consider (b). If $(aL-\sum_ia_iE_i)^2=(aL'-\sum_ia_iE_i')^2>0$,
we have $s(aL-\sum_ia_iE_i)\in\EFF(X_{r+1})$ for some $s\gg0$. Pick some effective divisor $C$
whose class is $s(aL-\sum_ia_iE_i)$. For each prime divisor component $D$
of $C$, there is an open set of points $p_i$ for which $D$ remains prime,
since being effective is a closed condition by (a), and
since for only finitely many classes $D'=a'L-\sum_ia_i'E_i$ could $D'$ and the class of $D-D'$ 
(or even $s(aL-\sum_ia_iE_i)-D'$ in place of $D-D'$) both
conceivably be classes of effective divisors. Thus the decomposition of 
$C$ as a sum of prime divisors is well-defined for general points,
and each component
specializes to an effective divisor on $X'_{r+1}$ which thus
meets $aL'-\sum_ia_iE_i'$, and hence $aL-\sum_ia_iE_i$, 
non-negatively, so $aL-\sum_ia_iE_i\in\NEF(X_{r+1})$, proving (b).
\end{proof}

Here is a version of the same result stated for generic points, where $X$ now is any surface and
$L$ comes via pullback from some ample divisor on $X$:

\begin{thm}[Semicontinuity Principle 2]\label{SCprin2} Let $a,a_1,\ldots,a_r$ be
integers, let $p_1,\ldots,p_r$ be generic points of a surface $X$
and denote by $X_{r+1}$ the blow 
up of $X$ at $p_1,\ldots,p_r$, with $E_1,\ldots,E_r$ being the usual associated classes
and $L$ the pullback to $X_{r+1}$ from $X$ of some ample divisor on $X$.
Also, given essentially distinct points $p_1',\ldots,p_r'$ of $X$, let $X'_{r+1}$ denote the blow 
up of $X$ at $p_1',\ldots,p_r'$, and let $L',E_1',\ldots,E_r'$ denote the associated classes.
\begin{itemize}
\item[(a)] If $aL-\sum_ia_iE_i\in\EFF(X_{r+1})$, 
then $aL'-\sum_ia_iE_i'\in\EFF(X'_{r+1})$ for every choice of essentially distinct points 
$p_1',\ldots,p_r'$ of $X$. 
\item[(b)] If $aL'-\sum_ia_iE_i'\in\NEF(X'_{r+1})$ for some choice of essentially distinct points 
$p_1',\ldots,p_r'$ of $X$, then $aL-\sum_ia_iE_i\in\NEF(X_{r+1})$.
\end{itemize}
\end{thm}

\begin{proof}
The proof of (a) is the same as for Theorem \ref{SCprin}(a). 
The proof for (b) is even simpler than before since now we are not claiming that
having a specific divisor $F=aL-\sum_ia_iE_i$ be nef is an open condition
on the points $p_i$. Instead, if $F$ were not nef, then $F\cdot C<0$ for some
$C=c_0L-\sum_ic_iE_i\in\EFF(X_{r+1})$, hence
$C'=c_0L'-\sum_ic_iE_i'\in\EFF(X_{r+1})$ by (a), so
$F'\cdot C'=F\cdot C<0$ for $F'=aL'-\sum_ia_iE_i'$, contradicting
our assumption that $F'\in \NEF(X'_{r+1})$. 
\end{proof}

\begin{Rmk} It is not hard to show that $F\in\NEF(X)$ implies $F^2\ge0$ (this is Exercise \ref{ex0}(a)). 
It is certainly possible, however, to have
$H\in\EFF(X)$ with $H^2<0$. The question of the extent to which this can happen is
the main motivation for these notes. If in fact there is no $H\in\EFF(X)$ with $H^2<0$,
then it is easy to see that $\EFF(X)\subseteq\NEF(X)$. It can also happen that
$\NEF(X)\subseteq\EFF(X)$, but in general neither containment holds.
For example, for $n>0$, the base curve $C$ on the Hirzebruch surface $H_n$
is effective but has $C^2=-n$ so is not nef. For an example of a nef divisor which
is not effective, see Exercise \ref{ex0}(b).
However, in Exercise \ref{ex0}(b), the class $F$ is in fact ample (see Exercise \ref{ex17}), 
thus some multiple of $F$ is effective (in fact $2F\in\EFF(X)$ by Riemann-Roch), but
divisors can be nef without being ample and without any multiple being effective.
For example, suppose $X$
is given by blowing up $r=s^2$ generic points $p_i\in\pr2$. Nagata  \cite{refN2} proved that
$h^0(X, \OO_X(mF))=0$ for all $m>0$ when
$F=sL-E_1-\cdots-E_r$ and $s>3$. But by specializing the points $p_i$
to general points of a smooth curve of degree $s$, we see that
$sL-E_1-\cdots-E_r$ is nef after specializing, and hence nef  to begin with by Theorem \ref{SCprin2}.
Thus for $r=s^2$ generic points $p_i$, $sL-E_1-\cdots-E_r$ is nef but 
not ample (since $F^2=0$), and, for each $m>0$,
$m(sL-E_1-\cdots-E_r)$ is not the class of an effective divisor. 
\end{Rmk}

\subsection{A Motivational Folklore Conjecture}\label{MFLC}

There is a long-standing open conjecture involving boundedness of negativity on surfaces. Let us say that 
a surface $X$ has {\it bounded negativity\/} if there is an integer $n_X$ such that $C^2\ge n_X$ for 
each prime divisor $C\subset X$.

\begin{Conj}[Folklore: The Bounded Negativity Conjecture]\label{folkloreconj} Every surface $X$ in characteristic 0
has bounded negativity.
\end{Conj}

\begin{Rmk} Conjecture \ref{folkloreconj} is false in positive characteristic.
I thank Burt Totaro for bringing to my attention the following example pointed out by J\'anos Koll\'ar
at a talk by Richard Harris at MSRI in January, 2009 (\url{http://www.msri.org/communications/vmath/VMathVideos/VideoInfo/4111/show_video}). Let $X=C\times C$, where $C$ is a curve of genus $g_C\ge2$ defined over a 
finite field of characteristic $p>0$. Let $\Gamma_q$ be the graph in $X$ of the Frobenius morphism
defined by taking $q$th powers, where $q$ is a sufficiently large power of $p$. Then $\Gamma_q$ is a 
curve on $X$ with $X^2=q(2-2g_C)$ \cite[Exercise V.1.10]{refHr}. Since $q$ can be arbitrarily large, 
$X$ does not have bounded negativity.
However, it is as far as I know still an open problem even in positive characteristic 
to determine which surfaces fail to have bounded negativity. 
\end{Rmk}

Some surfaces are known to have bounded negativity.

\begin{cor} A surface $X$ has bounded negativity if $-mK_X\in \EFF(X)$ for some positive integer $m$.
\end{cor}

\begin{proof} Since $-mK_X\in \EFF(X)$, there are only finitely many prime divisors $C$ such that 
$-mK_X\cdot C<0$. So, apart from finitely many prime divisors $C$,
we have $-mK_X\cdot C\ge 0$, in which case $C^2=2p_C-2-C\cdot K_X\ge -2$.
\end{proof}

\begin{Eg} In particular, bounded negativity holds for K3 surfaces, Enriques surfaces, abelian surfaces, 
and relatively minimal rational surfaces. But it is not always clear when it holds if one blows up points on those surfaces.
\end{Eg}

Let $\EFF(X)/\!\!\!\sim$ denote the image of $\EFF(X)$ in $\Num(X)$. In 
preparation for giving a criterion for bounded negativity to hold on $X$, we have
the following proposition (taken from \cite{refRo}):

\begin{prop}\label{fingeneffprop} If $\EFF(X)/\!\!\!\sim$ is finitely generated, then 
there are only finitely many prime divisors $C$ with $C^2<0$.
\end{prop}

\begin{proof} Let $C_1,\ldots,C_r$ be prime divisors whose classes generate $\EFF(X)/\!\!\!\sim$.
Since each $C_i$ is the class of an effective divisor, there are only finitely many prime divisors $D$
such that $D\cdot C_i<0$ for some $i$. Now let $C$ be a prime divisor with $C^2<0$; we have
$C\sim\sum_im_iC_i$ for some $m_i\ge0$ and so $0>C^2=\sum_im_iC\cdot C_i$, hence
$C\cdot C_j<0$ for some $j$, and so $C=C_j$.
\end{proof}

This then gives a criterion for bounded negativity to hold.
 
\begin{cor} If $X$ is a surface such that $\EFF(X)/\!\!\!\sim$ is finitely generated, then
bounded negativity holds for $X$.
\end{cor}

\begin{Rmk} It is difficult in general to determine whether $\EFF(X)/\!\!\!\sim$ is finitely generated,
even for rational surfaces.
Here are some cases where it is known. If $X$ is a rational surface with $K_X^2>0$,
then $\EFF(X)$ is finitely generated (see Exercise \ref{ex4} if $-K_X$ is nef
and $K_X^2>1$; see \cite{refRo} for the case that $K_X^2>0$ and $X$ is obtained 
by blowing up at most 8 points of $\pr2$; or see \cite{refLH} or
\cite{refTVV} for $K_X^2>0$ in general). We also have 
$\EFF(X)$ finitely generated
if $X$ is obtained by blowing up points on a line or conic in $\pr2$ (see 
Exercises \ref{ex1} and \ref{ex3}), or, more generally, 
if $X$ is rational and $-K_X$ is big \cite{refTVV}.
\end{Rmk}

When $X$ be obtained by blowing up $r<9$ generic points of $\pr2$, then $K_X^2>0$ and so
$\EFF(X)$ is finitely generated as mentioned above, but in fact $-K_X$ is ample,
which with adjunction implies $C^2\ge-1$ for any prime divisor $C$. For $r=9$,
$\EFF(X)$ is not finitely generated (see Exercise \ref{ex6}, for example), but
it is still true that $C^2\ge-1$ for any prime divisor $C$, although the proof is somewhat technical.
Here now is a conjecture for a case where $\EFF(X)$ is definitely not finitely generated
(see Exercise \ref{ex6})
but where Conjecture \ref{folkloreconj} is not yet known:

\begin{Conj}[\cite{refHa3}]\label{P2conj} Let $X$ be obtained by blowing up $r>9$ generic points
of $\pr2$. Then $C^2\ge-1$ for every prime divisor $C$, with equality if and only if
$C$ is a smooth rational curve with $K_X\cdot C=-1$.
\end{Conj}

\subsection{An Asymptotic Approach to Bounded Negativity}\label{asymp}

While no general lower bound for $C^2$ for prime divisors $C$ on a surface $X$ is known,
given a nef divisor $F$, we can instead ask for a lower bound on $C^2/(F\cdot C)^2$ for all prime divisors $C$
with $F\cdot C>0$. As motivation for introducing multipoint
Seshadri constants, we now study this question in the case that $X$ is obtained by blowing up $r>0$ points
$p_1,\ldots,p_r\in\pr2$, taking $F$ to be $L$. Since $C$ is prime and we assume $L\cdot C>0$, we see that
$C=dL-\sum_im_iE_i$ for some $m_i\ge0$. 

Let 
\addtocounter{thm}{1}
\begin{align*}\tag{\thethm}\label{altdefLambda}
\lambda_L(X)&=\inf\Big\{\frac{C^2}{(C\cdot L)^2}: C\in\EFF(X), C\cdot L>0,C\cdot E_i\ge0{\rm\ for\ all\ }i, C\cdot \sum_iE_i>0\Big\}\\
&=\inf\Big\{\frac{C^2}{(C\cdot L)^2}: C\hbox{\rm\ is a prime divisor on $X$ and\ } C\cdot L>0\Big\}
\end{align*}
(The second equality is Exercise \ref{ex6b}.)
It is clear that the infimum exists: Let $C=dL-\sum_im_iE_i$. Since $L-E_i\in\NEF(X)$, we see 
$d\ge m_i$ for all $i$. Hence $C^2/(C\cdot L)^2=(d^2-\sum_im_i^2)/d^2 \ge 1-r$.

\begin{Prob}\label{asympprob} Compute $\lambda_L(X)$, or at least give good estimates for it.
\end{Prob}

We now recall a quantity $\varepsilon$ introduced by G. V. Chudnovsky \cite{refCh}
(for any $r$ points in any projective space)
and Demailly \cite{refD} (for a single point, i.e., $r=1$, but on any smooth variety)
now known as a {\it multipoint Seshadri constant}; see also \cite{refLz}.
(Chudnovsky's version, denoted $\widehat{\Omega}_0(p_1,\ldots,p_r)$,
is actually equal to $r \varepsilon$.)
Let $X$ be obtained by blowing up distinct points $p_1,\ldots,p_r\in\pr2$. Then
$$\varepsilon(\pr2;p_1,\ldots,p_r)=\inf\Big\{\frac{d}{\sum_im_i}: dL-\sum_im_iE_i\in\EFF(X), m_i\ge0,\sum_im_i>0\Big\}$$

As alternative definitions (see Exercise \ref{ex7}) we have:
\addtocounter{thm}{1}
\begin{align*}\tag{\thethm}\label{altdefvareps}
\varepsilon(\pr2;p_1,\ldots,p_r) & =\inf\Big\{\frac{C\cdot L}{\sum_iC\cdot E_i}: C{\rm\ is\ prime\ and\ }\sum_iC\cdot E_i>0\Big\}\\
 & \\
                         & =\sup\Big\{\frac{m}{d}: dL-m\sum_iE_i\in\NEF(X), m>0\Big\}.
\end{align*}

\begin{Rmk}\label{easySCcomp}
In general, $\varepsilon(\pr2;p_1,\ldots,p_r)$ is itself hard to compute. However, by Exercise \ref{ex7b},
if $F\cdot C=0$ for some
$F=dL-m\sum_iE_i\in\NEF(X)$ and $C=aL-\sum_im_iE_i\in\EFF(X)$ with $d>0$ and $a>0$,
then $\varepsilon(\pr2;p_1,\ldots,p_r)=m/d$.
\end{Rmk}

For our asymptotic application of Seshadri constants to bounded negativity, we will use the following 
elementary inequality:

\begin{lem}
Given integers $d>0$ and $d\ge m_i\ge0$ for all $i$, we have $\displaystyle\frac{\sum_im_i^2}{d^2}\le \frac{\sum_im_i}{d}$.
\end{lem}

\begin{proof}
Just note that $d\ge m_i\ge0$ implies $d\sum_im_i\ge \sum_im_i^2$; dividing by $d^2$ gives the result.
\end{proof}

Since $L-E_i\in\NEF(X)$ for each $i$, if $dL-\sum_im_iE_i\in\EFF(X)$ then $d\ge m_i$ for all $i$,
hence $\frac{d}{\sum_im_i}\ge 1/r$. In particular, $\varepsilon(\pr2;p_1,\ldots,p_r)\ge 1/r>0$ so 
$1/\varepsilon(\pr2;p_1,\ldots,p_r)$ makes sense.
Applying the lemma now gives:

\begin{cor}\label{asympfllb} Let $X$ be obtained by blowing up distinct points $p_1,\ldots,p_r\in\pr2$. Then
$$\lambda_L(X)\ge 1-\frac{1}{\varepsilon(\pr2;p_1,\ldots,p_r)}.$$
\end{cor}

\begin{proof}
\begin{align*}
\lambda_L(X)&=\inf\Big\{\frac{C^2}{(C\cdot L)^2}: C\in\EFF(X), C\cdot L>0,C\cdot E_i\ge0{\rm\ for\ all\ }i, C\cdot \sum_iE_i>0\Big\}\\
&=\inf\Big\{1-\frac{\sum_im_i^2}{d^2}: C=dL-\sum_im_iE_i\in\EFF(X), d>0, m_i\ge0, \sum_im_i>0\Big\}\\
&=\inf\Big\{1-\frac{\sum_im_i^2}{d^2}: dL-\sum_im_iE_i\in\EFF(X), m_i\ge0,\sum_im_i>0\Big\}\\
&\ge\inf\Big\{1-\frac{\sum_im_i}{d}: dL-\sum_im_iE_i\in\EFF(X), m_i\ge0,\sum_im_i>0\Big\}\\
&=1-\sup\Big\{\frac{\sum_im_i}{d}: dL-\sum_im_iE_i\in\EFF(X), m_i\ge0,\sum_im_i>0\Big\}\\
&=1-\frac{1}{\varepsilon(\pr2;p_1,\ldots,p_r)}
\end{align*}

\end{proof}

\begin{Rmk}
Sometimes equality holds, but usually not. See Exercises \ref{ex8} and \ref{ex9}.
\end{Rmk}

We close this lecture with some remarks about Mori dream spaces. Mori dream spaces give
interesting examples of surfaces with bounded negativity.
Let $X$ be obtained by blowing up points $p_1,\ldots,p_r\in\pr2$.
Let $L$ be the total transform of a line and let $E_i$ be the blow up of the point $p_i$.
Define the Cox ring $\Cox(X)$ of $X$ to be the ring whose additive structure
is given by 
$$\bigoplus_{(a_0,\ldots,a_r)\in {\bf Z}^{r+1}}H^0(X,\OO_X(a_0L-a_1E_1-\cdots-a_eE_r)),$$
and where multiplication is given by the natural maps
$H^0(X, \OO_X(F_1))\otimes H^0(X, \OO_X(F_2))\to H^0(X, \OO_X(F_1+F_2))$.
If $\Cox(X)$ is finitely generated we say that
$X$ is a {\it Mori dream space\/} \cite{refHK}.
If $X$ is a Mori dream space, then $\EFF(X)$ must be finitely generated,
and hence $X$ has bounded
negativity by Proposition \ref{fingeneffprop}.

\begin{Rmk}
If $X$ is obtained by blowing up at most 8 points of $\pr2$, then
$X$ is a Mori dream space. (Proof: By the Hodge Index Theorem
if $F$ is a nontrivial nef divisor, then $-K_X\cdot F>0$, hence
the result follows over the complex numbers
from \cite[Corollary 1]{refGM} by \cite[Theorem 2]{refRo} and \cite[Theorem III.1]{refHa2}.)
In fact, if $X$ is any rational surface with $K_X^2>0$, then 
$X$ is a Mori dream space. (The same proof applies, but without the assumption of the complex numbers, 
using \cite[Proposition 4.3(a)]{refLH} in place of
\cite{refRo}; alternatively, see \cite{refTVV}.)
If $K_X^2=0$ but $-K_X$ is not nef, we can again conclude that $X$ is a Mori dream space.
(By \cite[Proposition 4.3(c)]{refLH}, $\EFF(X)$ is finitely generated, and by Exercise \ref{ex5}(b)
and \cite[Theorem III.1]{refHa2}, nef divisors are semi-ample (i.e., have a positive multiple
which is effective and base point free). Now
apply \cite[Corollary 1]{refGM}.) In each of these cases, $-K_X$ is big (see Exercise \ref{ex9a}),
hence these (in addition to the examples of Exercises \ref{ex2} and \ref{ex3} of blow ups of points on a line or conic)
are all subsumed by the result of \cite{refTVV} that a rational surface with big $-K_X$ is 
a Mori dream space. However, not all rational surfaces which are Mori dream spaces have big $-K_X$.
For example, let $C$ be an irreducible cubic curve, and blow up the curve $r>9$ times,
each time at successive infinitely near points of the cubic, starting with a flex point of the cubic.
By Exercise \ref{ex9b}, $\EFF(X)$ is finitely generated and any nef class $F$ has
$F\cdot(-K_X)\ge0$. By \cite[Theorem 3.1 and Corollary 3.4]{refHa6}, every nef class is semi-ample,
and by \cite{refGM}, $X$ is a Mori dream space since $\EFF(X)$ is finitely generated and
any nef class is semi-ample, but $-K_X$ is not big since it is a prime divisor of negative self-intersection.
\end{Rmk}

Here is a question I do not know the answer to:

\begin{Ques} If $X$ is a rational surface with $\EFF(X)$ finitely generated,
is $X$ a Mori dream space?
\end{Ques}

\vskip1in

\subsection{Exercises}

\begin{Ex}\label{ex0} Let $X$ be a surface. 
\begin{itemize}
\item[(a)] Show $F\in\NEF(X)$ implies $F^2\ge0$.
\item[(b)] Assume $X$ is obtained by blowing up $r=21$ general points $p_i\in\pr2$.
Then $\Cl(X)$ has basis $L, E_1,\ldots,E_{21}$, where
$L$ is the pullback of the class of a line and $E_i$ is the class of the blow up of $p_i$. Let $F=5L-\sum_iE_i$;
show that $F\in\NEF(X)\setminus\EFF(X)$.
\end{itemize}
\end{Ex}

\begin{Soln}[Exercise \ref{ex0}] (a) Let $A$ be ample, $F$ nef and $F^2<0$. 
We will show that there are positive integers $s$,
$a$ and $f$ such that $aA+fF$ is ample and $saA+sfF$ is effective, but such that $F\cdot(aA+fF)<0$, 
which is impossible if $F$ is nef.
To show $aA+fF$ is ample it is enough by the Nakai-Moisezon criterion \cite{refHr} 
to show that $(aA+fF)\cdot C>0$ for every curve $C$, and that $(aA+fF)^2>0$.
But $A\cdot C>0$ since $A$ is ample and $F\cdot C\ge0$ since $F$ is nef, so $(aA+fF)\cdot C>0$.
Since $A$ is ample, $aA\in\EFF(X)$ for $a\gg0$, hence $A\cdot F\ge0$. Thus, taking $t=f/a$, we have
$(aA+fF)\cdot F=a(A+tF)\cdot F<0$ for $t>\frac{A\cdot F}{-F^2}$, but
$a(A+tF)\cdot F=0$ and
$(aA+fF)^2=a^2(A^2+2tA\cdot F+t^2F^2)=a^2(A^2+tA\cdot F)>0$ for  
$t=\frac{A\cdot F}{-F^2}$, so by choosing $a$ and $f$ such that $t$ is slightly larger than 
$\frac{A\cdot F}{-F^2}$ we will still have $(aA+fF)^2>0$ (and hence $aA+fF$ is ample
so $saA+sfF$ is effective for $s\gg0$) while also
having $(saA+sfF)\cdot F<0$, contradicting $F$ being nef. Hence we must have $F^2\ge0$.

(b) Consider points $p'_i$ which lie on a smooth quintic.
Let $Q$ be the proper transform of that quintic. Then $Q$ is nef
but $Q$ is linearly equivalent to $F'=5L'-E'_1-\cdots-E'_{21}$.
Now by the semicontinuity principle, Theorem \ref{SCprin},
$F=5L-E_1-\cdots-E_{21}\in\NEF(X)$ when the points $p_i$ are general.
But the points are general so impose 21 independent conditions on 
the 21 dimensional space of all quintics (since we can always
choose each successive point not to be a base point of the linear system of quintics 
through the previous points); i.e., $h^0(X,\OO_X(F))=0$ hence $F\not\in\EFF(X)$. 
\end{Soln}

\begin{Ex}\label{ex1} Find an explicit finite set of generators for $\EFF(X)$ and $\NEF(X)$ 
in case $X$ is obtained by blowing up $r\ge 1$ distinct points on a line in $\pr2$.
\end{Ex}

\begin{Soln}[Exercise \ref{ex1}] This solution is based on \cite[Proposition I.5.2]{refHa3}.
Let the points be $p_1,\ldots,p_r$. Then $\Cl(X)$ has basis $L, E_1,\ldots,E_r$, where
$L$ is the pullback of the class of a line and $E_i$ is the class of the blow up of $p_i$.
Let $\Lambda=L-E_1-\cdots-E_r$ and let $L_i=L-E_i$.
Clearly $\Lambda\in\EFF(X)$ and $E_i$ and $L_i$ are in $\EFF(X)$ for each $i$.
Since $L_i^2=0$ and $L_i$ is the class of a prime divisor, we see $L_i\in\NEF(X)$.
To prove that $\Lambda, E_1,\ldots,E_r$ generate $\EFF(X)$,
it is enough to prove that every effective, reduced, irreducible
divisor can be written as a non-negative integer combination of $\Lambda$
and $E_1,\dots,E_r$. So let $C=a\Lambda+\sum b_iE_i$
be the class of an effective, reduced and irreducible divisor. 
If $C$ is $\Lambda$ or $L_j$, then the claim
is true as $L_j=\Lambda+\sum_{i\neq j}E_i$, so we may assume
that $C$ is not one of these divisors. But then the intersection
with them must be non-negative and hence $b_j=C\cdot L_j\geq 0$.
Putting this into $0\leq C\cdot\Lambda=a-\sum b_i$
implies the non-negativity of $a$.
Moreover, if $F$ is nef then $F$ meets each $E_i$ and $\Lambda$ non-negatively,
and the argument we just used on $C$ shows that any such class can be written 
as $F=(a-\sum_ib_i)L+\sum_ib_iL_i$ for non-negative integers $a,b_1,\ldots,b_r$,
and hence $L,L_1,\ldots,L_r$ generate $\NEF(X)$. 
\end{Soln}

\begin{Ex}\label{ex2} Let $X$ be obtained by blowing up points $p_1,\ldots,p_r\in\pr2$.
If the points $p_i$ are collinear, show that $X$ is a Mori dream space.
\end{Ex}

\begin{Soln}[Exercise \ref{ex2}] The generators are given by taking a basis for
$H^0(X,\OO_X(G))$ for each $G$ among $E_1,\ldots,E_r$, $\Lambda$ and
$L_1,\ldots,L_r$. This is because if $D$ is an effective divisor, then by the solution
to Exercise \ref{ex1}, $D=N+M$, where $N$ is fixed and consists of a sum
of non-negative multiples of the $E_i$ and $\Lambda$, and $M\in\NEF(X)$
and hence $M=\sum_{i\ge0}m_iL_i$ for some non-negative $m_i$ (where 
we take $L_0=L$).
Thus it is enough to show that $\bigotimes_i H^0(X, \OO_X(L_i))^{\otimes m_i}\to H^0(X, \OO_X(M))$
is surjective. Do this inductively by showing that 
$H^0(X, \OO_X(F))\otimes H^0(X, \OO_X(L_i))\to H^0(X, \OO_X(F+L_i))$
is surjective for each $F\in\NEF(X)$ and hence in fact that
$H^0(X, \OO_X(F_1))\otimes H^0(X, \OO_X(F_2))\to H^0(X, \OO_X(F_1+F_2))$
is surjective whenever $F_1$ and $F_2$ are nef (see \cite[Theorem 2.8]{refHa1}). 
Alternatively, see  \cite{refOt}.
\end{Soln}

\begin{Ex}\label{ex3}
Let $X$ be obtained by blowing up points $p_1,\ldots,p_r$ on a smooth conic in $\pr2$
with $r\ge 3$. (If $r<3$, the points are collinear and the result is given by Exercise \ref{ex1}.
Also, the conic does not need to be smooth here but smoothness simplifies the argument a bit.)
\begin{itemize}
\item[(a)] Show $\EFF(X)$ is finitely generated. 
\item[(b)] Cite the literature to show that $X$ is a Mori dream space.
\end{itemize}
\end{Ex}

\begin{Soln}[Exercise \ref{ex3}] (a) Let $L_{ij}$, $i\ne j$, be the class of the proper transform of the line through
$p_i$ and $p_j$, let $L$ be the class of the total transform of a line, let $D$ be the 
class of the proper transform of the conic
and let $E_i$ be the class of the blow up of $p_i$ for each $i>0$. 

Let $C$ be the class of a prime divisor. Note that $-K_X=D+L$.
Hence if $C^2<0$ but $C\ne D$, then adjunction forces $L\cdot C\le 1$,
and hence $C$ is either $L_{ij}$ or $E_i$ for some $i$ and $j$.
If $C^2\ge 0$, write $C=a_0L-a_1E_1-\cdots-a_rE_r$. 
Since $C\cdot E_i\ge0$, we have $a_i\ge0$ for all $i$.
By reindexing we may assume that $a_1\ge a_2\ge \cdots\ge a_r\ge0$.
Let $D_i=2L-E_1-\cdots-E_i=D+E_{i+1}+\cdots+E_r$.
Thus $C=(a_0-a_1-a_3)L+(a_1-a_2)(L-E_1)+(a_2-a_3)L_{12}
+(a_3-a_2)D_3+\cdots+(a_{r-1}-a_r)D_{r-1}+a_rD_r$,
where $a_0-a_1-a_3\ge0$ since $C\cdot L_{12}\ge0$ implies
$a_0\ge a_1+a_2\ge a_1+a_3$. Note that $L=L_{12}+E_1+E_2$ and that $L-E_1=L_{12}+E_2$.
In particular, the class of every prime divisor is a sum of 
non-negative multiples of classes of the form $L_{ij}$, $E_i$ and $D$.

(b) Since $-K_X=D+L$, if $F$ is nef with
$-K_X\cdot F=0$, then $F\cdot L=0$, hence $F=0$ by the Hodge Index Theorem.
The fact that $X$ is a Mori dream space now follows by
\cite[Corollary 1]{refGM} and \cite[Theorem III.1]{refHa2},
or directly by \cite[Corollary 3]{refGM}, or by \cite{refTVV}.
\end{Soln}

The basic idea of part (a) of the next exercise is taken from \cite{refRo}.
\vskip\baselineskip

\begin{Ex}\label{ex4} Let $X$ be a rational surface such that $-K_X$ is nef.
\begin{itemize}
\item[(a)] If $C$ is a prime divisor on $X$ such that $p_C>0$, show that $C+K_X\in\EFF(X)$.
\item[(b)] For each integer $n$, show that there are only finitely many classes C of prime divisors
with $C^2\le n$ if $K_X^2>0$. 
\item[(c)] If $E$ is a class such that $E^2=E\cdot K_X=-1$, show that $E\in\EFF(X)$.
\item[(d)] If $K_X^2>1$, show that $-K_X-E\in\EFF(X)$ for any class $E$ such that $E^2=E\cdot K_X=-1$.
\item[(e)] Conclude that $\EFF(X)$ is finitely generated if $K_X^2>1$.
\end{itemize}
\end{Ex}

\begin{Soln}[Exercise \ref{ex4}] 
First note that $K_X^2\ge0$ since $-K_X$ is nef. 
Next note that $h^2(X,\OO_X(-K_X))=h^0(X,\OO_X(2K_X))$ is 0 since $X$ is rational.
Hence $h^0(X,\OO_X(-K_X))\ge K_X^2+1>0$ by Riemann-Roch, so $-K_X\in\EFF(X)$.

(a) Take cohomology of 
$0\to \OO_X(K_X)\to \OO_X(C+K_X)\to \OO_C(C+K_X)\to 0$.
Since $X$ is rational, $h^1(X, \OO_X(K_X))=h^1(X, \OO_X)=0$,
and $h^0(X,\OO_X(K_X))=0$ since $-K_X\in\EFF(X)$ is nontrivial.
Thus $0<p_C=h^0(C,\OO_C(C+K_X))=h^0(X,\OO_X(C+K_X))$.

(b) By adjunction and the fact that $-K_X$ is nef we have
$C^2=2p_C-2-K_X\cdot C\ge -2$, so for each $n$ it is enough to show that there are
only finitely many $C$ with $C^2=n$. So say $C^2=n$, hence 
$0\le -K_X\cdot C\le C^2+2=n+2$ by adjunction and the fact that $-K_X$ is nef. 
Now let $N= K_X^2C-(K_X\cdot C)K_X$, so $C=((-K_X\cdot C)(-K_X)+N)/(K_X^2)$.
Thus to show there are only finitely many such $C$, it is enough
to show that $-K_X\cdot C$ is bounded (but we already saw that
$0\le -K_X\cdot C\le n+2$) and that there are only finitely
many possibilities for $N$. To see the latter, note that $N\cdot K_X=0$, so
$(K_X^2)^2C^2=(K_X\cdot C)^2K_X^2+N^2$, hence 
$(K_X^2)(K_X^2n-(n+2)^2)\le (K_X^2)^2C^2-(K_X\cdot C)^2K_X^2=N^2$.
Thus $N^2$ is bounded below,
but $N\in K_X^{\perp}$ and $K_X^{\perp}$ is negative definite by the Hodge Index Theorem, 
so intuitively there are only finitely many lattice elements $N$ in $K_X^{\perp}$
of length at most $\sqrt{-(K_X^2)(K_X^2n-(n+2)^2)}$. More rigorously,
since $\Cl(X)$ is free abelian of finite rank, there are only finitely many elements of
$\Cl(X)$ orthogonal to $K_X$ with self-intersection no less than $(K_X^2)(K_X^2n-(n+2)^2)$,
and hence there are only finitely many possibilities for $N$.

(c) Since $-K_X$ is nef but $-K_X\cdot (K_X-E)<0$, we see $K_X-E\not\in \EFF(X)$.
Thus $h^2(X,\OO_X(E))=0$, hence $h^0(X,\OO_X(E))\ge 1$ by Riemann-Roch,
so $E\in\EFF(X)$. 

(d) Since $-K_X\cdot (2K_X+E)<0$, we see that $0=h^0(X,\OO_X(2K_X+E))=h^2(X,\OO_X(-K_X-E))$.
Now $h^0(X,\OO_X(-K_X-E))\ge K_X^2-1$ by Riemann-Roch, so $-K_X-E\in \EFF(X)$.

(e) Let $E$ be any class such that $E^2=K_X\cdot E=-1$. Then $E\in\EFF(X)$.
Let $C$ be the class of a prime divisor that is a component of $E$.
Then $p_C=0$, since otherwise $C+K_X$ and hence
$E+K_X$ is in $\EFF(X)$ by (a). But $-K_X\cdot(E+K_X)<0$, so this is impossible.
Thus $E$ is a sum of prime divisors $C$ with $p_C=0$. Likewise, $-K_X-E$ is a sum of 
classes of prime divisors with $p_C=0$,
since otherwise $-K_X-E+K_X=-E$ is in $\EFF(X)$. So for some $C_i$ with $p_{C_i}=0$
we have $-K_X=\sum_iC_i$ and this sum involves at least two summands.

By part (b), there are only finitely many classes $D$ of prime divisors with $D^2\le0$ and $p_D=0$.
We will now see that these classes $D$, together with the $C_i$, generate $\EFF(X)$.
Given any prime divisor $C$, it is enough to show that either $C-D\in\EFF(X)$ for some
such $D$ or that $C-C_i\in\EFF(X)$ for some $i$. This is clear by (a) if $p_C>0$, so assume
$p_C=0$. It is again clear if $C^2\le0$, so assume $C^2>0$. We may assume
$C\cdot C_1\le C\cdot C_i$ for all $i$. Note that $h^2(X,\OO_X(C-C_1))=h^0(X,\OO_X(K_X-C+C_1))
=h^0(X,\OO_X(-\sum_{i>1}C_i-C))=0$. Thus $h^0(X,\OO_X(C-C_1))\ge 
(C^2-2C\cdot C_1-K_X\cdot C + (C_1^2+K_X\cdot C_1))/2+1=
(C^2-2C\cdot C_1-K_X\cdot C -2)/2+1= (C^2+\sum_{i>1}C\cdot C_i-C\cdot C_1)/2 \ge C^2/2>0$,
so $C-C_1\in\EFF(X)$. 
\end{Soln}

\begin{Ex}\label{ex5}
Let $X$ be a rational surface with $K_X^2=0$.
\begin{itemize}
\item[(a)] Show that $-K_X\in\EFF(X)$.
\item[(b)] Assume in addition that $-K_X$ is not nef. If $F$ is nef
with $-K_X\cdot F=0$, show that $F=0$.
\end{itemize}
\end{Ex}

\begin{Soln}[Exercise \ref{ex5}] (a) Apply Riemann-Roch, using $h^2(X,\OO_X(-K_X))=h^0(X,\OO_X(2K_X))=0$.

(b) This follows by the Hodge Index Theorem.
Suppose $F\ne0$. Since $F$ is nef, we have $F^2\ge0$ (by Exercise
\ref{ex0}). If $F^2>0$, then $F^{\perp}$ is negative definite, hence
$-K_X\cdot F=0$ and $K_X^2=0$ imply that $-K_X=0$, but
this contradicts the fact that $-K_X\ne0$. Thus $F^2=0$. 
Since $F^2=0$ and $K_X\cdot F=0$, we see for any elements $v$ and $w$ 
in the span of $-K_X$ and $F$ in $\Cl(X)$ that $v\cdot w=0$. 
But for any ample divisor $A$ we have $i=A\cdot(-mK_X+F)>0$ for $m\gg0$,
since $j=-K_X\cdot A>0$.
Let $v=j(-mK_X+F)$ and let $w=-iK_X$. Then $v-w\in A^{\perp}$, but $(v-w)^2=0$,
so $v=w$, hence $j(-mK_X+F)=-iK_X$ so $jF=(mj-i)K_X$. Thus $mj-i<0$ 
(since $-3(mj-i)=(mj-i)K_X\cdot L=jF\cdot L\ge 0$ but $F\cdot L=0$ implies $F=0$) so $-K_X$ is nef
(being a positive rational  multiple of a nef class),
contrary to hypothesis.
\end{Soln}

\begin{Ex}\label{ex6}
Let $X$ be obtained by blowing up 9 points $p_1,\ldots,p_9\in\pr2$ on a smooth plane cubic $D'$.
Let $D$ be the proper transform of $D'$. Let $L$ be the pullback of the class of a line
and let $E_i$ be the class of the blow up of $p_i$.
\begin{itemize}
\item[(a)] Show that $N\in K^{\perp}$ implies $N^2$ is even.
\item[(b)] Let $N$ be any class in $K^{\perp}\cap E_9^{\perp}$.
Show that $E=N+E_9+(N^2/2)K_X$ satisfies $E^2=K_X\cdot E=-1$.
Conclude that $E\in\EFF(X)$.
\item[(c)] If the points are sufficiently general and the ground field is the complex numbers,
show that each such $E$ is the class of a prime divisor.
Conclude that $\EFF(X)$ is not finitely generated.
\end{itemize}
\end{Ex}

\begin{Soln}[Exercise \ref{ex6}] 
(a) By Riemann-Roch, $N^2/2=(N^2-K_X\cdot N)/2$ is an integer.

(b) That $E=N+E_9+(N^2/2)K_X$ satisfies $E^2=K_X\cdot E=-1$ is easy.
Since $-K_X=D$ is nef, and since $-K_X\cdot(K_X-E)<0$,
we see that $h^2(X,\OO_X(E))=0$. Now apply Riemann-Roch to see that 
$h^0(X,\OO_X(E))\ge 1$, so $E\in\EFF(X)$.

(c) Suppose that some $E=N+E_9+(N^2/2)K_X$ is not the class of a prime divisor. 
Since $-K_X\cdot E=1$ and $-K_X$ is nef, if $E$ has two 
or more components, then one of them must be disjoint from $D$, hence in the kernel of
the mapping $\Cl(X)\to \Cl(D)$. But the kernel here is the same as the kernel of
$K_X^{\perp}\to \Cl^0(D)$, where $\Cl^0(D)$ is the subgroup of divisor classes of degree 0,
which is a torus which can be identified with $D$. Since the complex numbers have infinite dimension
over the rationals, it's easy to choose points $p_i\in D$, such that
the map $K_X^{\perp}\to \Cl^0(D)$ is injective, hence no prime divisor on $X$ is disjoint from $D$,
so $E$ must be prime.
Conclude by applying Proposition \ref{fingeneffprop}.
\end{Soln}

\begin{Ex}\label{ex6b} Let $X$ be obtained by blowing up $r>0$ points
$p_1,\ldots,p_r\in\pr2$, and let $L$ and $E_i$ be as usual.
Show that
\begin{align*}
&\inf\Big\{\frac{C^2}{(C\cdot L)^2}: C\hbox{\rm\ is a prime divisor on $X$ and\ } C\cdot L>0\Big\}\\
=&\inf\Big\{\frac{C^2}{(C\cdot L)^2}: C\in\EFF(X), C\cdot L>0,C\cdot E_i\ge0{\rm\ for\ all\ }i, C\cdot \sum_iE_i>0\Big\}.
\end{align*}
\end{Ex}

\begin{Soln}[Exercise \ref{ex6b}] Let $\inf_1$ be the first infimum in the statement above
and let $\inf_2$ be the second. Note that $\inf_1$ is equal to
$\inf_3=\inf\Big\{\frac{C^2}{(C\cdot L)^2}: C\hbox{\rm\ is a prime divisor on $X$,\ } C\cdot L>0{\rm\ and\ }C\cdot \sum_iE_i>0\Big\}$, since any prime $C$ with $C\cdot L>0$ but $C\cdot \sum_iE_i=0$ is a
positive multiple of $L$, in which case $C'=C-E_1$ is the class of a prime divisor with 
$\frac{C'^2}{(C'\cdot L)^2}< \frac{C^2}{(C\cdot L)^2}$. 

Since $\inf_2$ is an infimum over a bigger set than is $\inf_3$, we see that
$\inf_2\le \inf_3=\inf_1$. Thus, to see $\inf_2=\inf_1$, it is enough to see for any
$D$ that is effective with $D\cdot L>0$, $D\cdot E_i\ge0$ for all $i$ and $D\cdot \sum_iE_i>0$
that there is a prime $C$ with $C\cdot L>0$ and $\frac{C^2}{(C\cdot L)^2}\le \frac{D^2}{(D\cdot L)^2}$.

Suppose $D$ satisfies the given conditions. Write $D$ as $dL-\sum_im_iE_i$. We can also
write $D$ as $\sum_jC_j$ for some prime divisors $C_j$.
Let $F$ be obtained by deleting every summand $C_j$ (if any) for which
$C_j=E_i$ for some $i$. Writing $F=dL-\sum_im'_iE_i$ we see $m'_i\ge m_i$ for all $i$,
hence $\frac{F^2}{(F\cdot L)^2}\le \frac{D^2}{(D\cdot L)^2}$.
Thus we may assume that $C_j\ne E_i$ for all $i$ and $j$ and hence that 
$C_j\cdot E_i\ge0$ for all $i$ and $j$ and that $C_j\cdot L\ge0$ for all $j$.
If for some $j$ we have $C_j\cdot \sum_iE_i=0$, then $C_j$ is a positive multiple of
$L$, so we can replace $C_j$ by $C_j-E_1$; the latter is still the class of a prime divisor,
but this change reduces $\frac{F^2}{(F\cdot L)^2}$.
Thus with these changes we may assume each summand $C_j$ 
of $F$ satisfies the conditions imposed on $D$.

By induction it is clearly enough to check that if $D_1$ and $D_2$ satisfy the conditions on $D$, then 
$$\min_i\Big\{\frac{D_i^2}{(D_i\cdot L)^2}\Big\}\le \frac{(D_1+D_2)^2}{((D_1+D_2)\cdot L)^2}.$$
If we write $D_1=aL-\sum_ia_iE_i$ and $D_2=bL-\sum_ib_iE_i$, and assume that
the minimum occurs for $i=1$,
this is just $(a^2-\sum_ia_i^2)/a^2\le ((a+b)^2-\sum_i(a_i+b_i)^2)/(a+b)^2$,
or $1-\sum_i(a_i/a)^2\le 1-\sum_i((a_i+b_i)/(a+b))^2$.
I.e., it is enough to show that $\sum_i(a_i/a)^2\ge\sum_i((a_i+b_i)/(a+b))^2$ if
$\sum_i(a_i/a)^2\ge\sum_i(b_i/b)^2$. I.e., given vectors $v$ and $w$ in Euclidean 
space with non-negative entries and given positive reals $a$ and $b$,
we must show $(v/a)^2\ge ((v+w)/(a+b))^2$ if $(v/a)^2\ge (w/b)^2$.

But $b^2v^2\ge a^2w^2$ by hypothesis, so $b|v|\ge a|w|$, hence $bv^2= b|v|^2\ge a|w||v|\ge av\cdot w$,
so $2abv^2\ge2a^2v\cdot w$ and thus 
$(a+b)^2v^2=a^2v^2+b^2v^2+2abv^2\ge a^2v^2+a^2w^2+2a^2v\cdot w=a^2(v+w)^2$,
whence $(v/a)^2\ge ((v+w)/(a+b))^2$.
\end{Soln}

\begin{Ex}\label{ex7}
Let $X$ be obtained by blowing up points $p_1,\ldots,p_r\in\pr2$.
\begin{itemize}
\item[(a)] Show that $\varepsilon(\pr2;p_1,\ldots,p_r)=\inf\Big\{\frac{C\cdot L}{\sum_iC\cdot E_i}: C{\rm\ is\ prime\ and\ }C\cdot L>0\Big\}$.
\item[(b)] Show that $\varepsilon'(\pr2;p_1,\ldots,p_r)= \varepsilon(\pr2;p_1,\ldots,p_r)$, where $\varepsilon'(\pr2;p_1,\ldots,p_r)=\sup\Big\{\frac{m}{d}: dL-m\sum_iE_i\in\NEF(X), d>0\Big\}$.
\end{itemize}
\end{Ex}

\begin{Soln}[Exercise \ref{ex7}] (a) This just amounts to the easy fact that
$$\frac{d_1+d_2}{m_1+m_2}\ge\min\Big\{\frac{d_1}{m_1},\frac{d_2}{m_2}\Big\}.$$
By definition
$$\varepsilon(\pr2;p_1,\ldots,p_r)=\inf\Big\{\frac{d}{\sum_im_i}: dL-\sum_im_iE_i\in\EFF(X), m_i\ge0,\sum_im_i>0\Big\}.$$
Suppose $D=dL-\sum_im_iE_i\in\EFF(X)$ satisfies the conditions of the definition of
$\varepsilon(\pr2;p_1,\ldots,p_r)$. Write $D=\sum_iC_i$ as a sum of classes of prime divisors $C_i$.
Deleting all $C_i$ of the form $E_j$ reduces $\frac{d}{\sum_im_i}$, so we may assume
$C_i\ne E_j$ for all $i$ and $j$. If for some $i$ we have $C_i\cdot E_j=0$ for all $j$,
then deleting that $C_i$ from the sum also reduces $\frac{d}{\sum_im_i}$.
Hence we may assume that $D=dL-\sum_im_iE_i=\sum_iC_i$, where each $C_i$ is prime and 
satisfies the conditions in the definition of $\varepsilon(\pr2;p_1,\ldots,p_r)$. Write
$C_i=d_iL-\sum_jm_{ij}E_j$. Let $\mu_i=\sum_jm_{ij}$. Then it suffices to show that
$(\sum_id_i)/(\sum_i\mu_i)\ge \min_i\{d_i/\mu_i\}$, which follows by repeated application of the easy fact above.

(b) If $aL-b\sum_iE_i\in\NEF(X)$, then $ad\ge b\sum_im_i$ whenever $dL-\sum_im_iE_i\in\EFF(X)$.
Thus 
$$\varepsilon(\pr2;p_1,\ldots,p_r)\ge \varepsilon'(\pr2;p_1,\ldots,p_r).$$ 
Conversely, for any positive integers $a$ and $b$ such that
$b/a<\varepsilon(\pr2;p_1,\ldots,p_r)$, we have $(aL-b\sum_iE_i)\cdot (dL-\sum_im_iE_i)\ge0$ for
all $dL-\sum_im_iE_i\in\EFF(X)$, and hence $aL-b\sum_iE_i\in\NEF(X)$ so
$b/a\le \varepsilon'(\pr2;p_1,\ldots,p_r)$. Since we can choose positive integers $a$ and 
$b$ such that $b/a$ is less than (but arbitrarily
close to) $\varepsilon(\pr2;p_1,\ldots,p_r)$, the result follows.
\end{Soln}

\begin{Ex}\label{ex7b} Let $X$ be the blow up of $\pr2$ at $r$ distinct points
$p_1,\ldots,p_r$. Suppose $F\cdot C=0$ for some
$F=dL-m\sum_iE_i\in\NEF(X)$ and $C=aL-\sum_im_iE_i\in\EFF(X)$ with $d>0$ and $a>0$.
Show that $\varepsilon(\pr2;p_1,\ldots,p_r)=m/d$.
\end{Ex}

\begin{Soln}[Exercise \ref{ex7b}] 
Since $F\in\NEF(X)$, clearly $\varepsilon(\pr2;p_1,\ldots,p_r)\ge m/d$.
But $\varepsilon(\pr2;p_1,\ldots,p_r) > m/d$ would imply that 
$F'=d'L-m'\sum_iE_i\in\NEF(X)$ for some $m'/d'>m/d$, but in that case
$F'\cdot C<0$, contradicting $F'\in\NEF(X)$.
\end{Soln}

\begin{Ex}\label{ex8}
Let $X$ be obtained by blowing up collinear points $p_1,\ldots,p_r\in\pr2$.
Show that 
$$\varepsilon(\pr2;p_1,\ldots,p_r)=1/r$$ 
and that equality holds in Corollary \ref{asympfllb}.
\end{Ex}

\begin{Soln}[Exercise \ref{ex8}] Since $C=L-E_1-\cdots-E_r$ is the class of a prime divisor
and $F=rL-E_1-\cdots-E_r=(r-1)L+C$ is a sum of prime divisors each of which $F$ meets non-negatively,
we see that $F$ is nef. Clearly $\lambda_L(X)\le C^2/(C\cdot L)^2=1-r$.
But  $F\cdot C=0$, so
$\varepsilon(\pr2;p_1,\ldots,p_r)=1/r$ by Exercise \ref{ex7b},
and we have $1-r=1-1/\varepsilon(\pr2;p_1,\ldots,p_r)$.
\end{Soln}

\begin{Rmk*}[on Exercise \ref{ex8}] Exercise \ref{ex8}  shows that equality holds in Corollary \ref{asympfllb}
when the points are collinear, but the converse is not true. Here is an example
where equality holds but the points are not collinear.
Suppose we consider 9 points on a smooth cubic, three of which are collinear.
Let $X$ be obtained by blowing up the nine points. The proper transform of the cubic
is clearly effective and (being prime of non-negative self-intersection) it also
is nef so $\varepsilon(\pr2;p_1,\ldots,p_r)=1/3$ by Exercise \ref{ex7b}, and the proper transform of the line
through the three collinear points has self-intersection $-2$, so $-2\ge \lambda_L(X)$, hence
$-2\ge \lambda_L(X)\ge1-1/\varepsilon(\pr2;p_1,\ldots,p_r)=-2$.

There is another way to look at what Exercise \ref{ex8} tells us, however.
The solution to Exercise \ref{ex8} shows that
$\lambda_L(X)=1-r$ and $\varepsilon(\pr2;p_1,\ldots,p_r)=1/r$
if the points are collinear. Conversely, if either
$\lambda_L(X)=1-r$ or $\varepsilon(\pr2;p_1,\ldots,p_r)=1/r$, then the points
$p_1,\ldots,p_r\in\pr2$ are collinear. For suppose $\varepsilon(\pr2;p_1,\ldots,p_r)=1/r$. 
In any case, $F=(r-1)L-E_1-\cdots-E_r=(L-E_1-E_2)+(L-E_3)+\cdots+(L-E_r)$
is effective and the classes of the prime components of $L-E_1-E_2$ consist of classes $E_j$
and $L-E_1-E_2-E_{j_1}-\cdots-E_{j_s}$, where $p_1,p_2,p_{j_1},\ldots,p_{j_s}$ are all of the points
which lie on the line through $p_1$ and $p_2$. If the points
were not all collinear, then there would be at most $r-1$ such points,
so $F$ would meet each of its prime components non-negatively. Thus $F$ would be nef and 
we would have the contradiction that $1/r=\varepsilon(\pr2;p_1,\ldots,p_r)\ge 1/(r-1)$. 
Finally, suppose $\lambda_L(X)=1-r$. Since $F=rL-E_1-\cdots-E_r=(L-E_1)+(L-E_2)+(L-E_3)+\cdots+(L-E_r)$
is always nef, we see that $\varepsilon(\pr2;p_1,\ldots,p_r)\ge1/r$ always holds.
But this means we have $1-r=\lambda_L(X)\ge 1-1/\varepsilon(\pr2;p_1,\ldots,p_r)\ge 1 - 1/(1/r) = 1-r$,
hence $\varepsilon(\pr2;p_1,\ldots,p_r)=1/r$ which we saw above implies the points are collinear.
\end{Rmk*}

\begin{Ex}\label{ex9}
Find a set of points $p_1,\ldots,p_r\in\pr2$ such that
the inequality in Corollary \ref{asympfllb} is strict.
\end{Ex}

\begin{Soln}[Exercise \ref{ex9}] Consider ten points $p_1,\ldots,p_{10}$ on a smooth conic.
From the solution to Exercise \ref{ex3}, the only prime divisors $C$ of 
negative self-intersection come from the points, from the lines through pairs of points
and from the conic itself. 
The infimum defining $\lambda_L(X)$ must come from prime divisors $C$ of negative self-intersection.
By just checking the possibilities we see $\lambda_L(X)=-6/4$ comes from $C=2L-E_1-\cdots-E_{10}\in\EFF(X)$
and we also see $F=5L-E_1-\cdots-E_{10}=C+3L$ is nef.
Since $F\cdot C=0$, by Exercise \ref{ex7b} we see $\varepsilon(\pr2;p_1,\ldots,p_r)=1/5$
and hence $\lambda_L(X)=-6/4\ge 1-1/\varepsilon(\pr2;p_1,\ldots,p_r)=-4$.
\end{Soln}

\begin{Ex}\label{ex9a}
Let $X$ be a rational surface such that $K_X^2=0$ but $-K_X\not\in\NEF(X)$.
Show that $-K_X$ is big
(i.e., some positive multiple $-mK_X$ is effective and can be written 
as $-mK_X=M+N$ where $M$ and $N$ are effective and $M^2>0$).
\end{Ex}

\begin{Soln}[Exercise \ref{ex9a}] 
By Exercise \ref{ex5}, $-K_X\in\EFF(X)$. Since $-K_X$ is not nef, there is a prime divisor $C$
such that $-K_X\cdot C<0$. Thus $-K_X-C$ is effective hence so is $-mK_X-C$ for $m\ge1$, and 
$(-mK_X-C)^2=(2m-1)K_X\cdot C +(K_X\cdot C+C^2)
=(2m-1)K_X\cdot C+2p_C-2$, so $(-mK_X-C)^2>0$ for $m\ge2$.
Thus $-mK_X=(-mK_X-C)+C$ is big for $m\ge2$. 
\end{Soln}

\begin{Ex}\label{ex9b}
Let $X=X_{r+1}$ be the rational surface such that
$X_1=\pr2$, and for each $i\ge 1$, $X_{i+1}\to X_i$ is the blow up
of $p_i$, where $p_1\in X_1$ is a flex of an irreducible plane cubic $C$, 
and then for each $i\ge1$, $p_{i+1}$ is the point of the proper transform of $C$
on $X_{i+1}$ infinitely near to $p_i$. (Thus $p_1,\ldots, p_r$ are 
essentially distinct points.)  Assume $r\ge 3$. Show that the class of any prime divisor $D$ with $D^2<0$
is either $E_r$, or $E_i-E_{i+1}$ for $1\le i<r$, or $L-E_1-E_2-E_3$
or $D=-K_X$ (if $r>9$), and show that a divisor class $F$ is nef if and only if $-K_X\cdot F\ge0$
and $F$ is a non-negative integer linear
combination of $L$, $L-E_1$, $2L-E_1-E_2$, $3L-E_1-E_2-E_3$, $\ldots$,
$3L-E_1-\cdots-E_r=-K_X$. Conclude that $\EFF(X)$ is generated by
$E_r$, $E_i-E_{i+1}$ for $1\le i<r$, $L-E_1-E_2-E_3$
and $D=-K_X$.
\end{Ex}

\begin{Soln}[Exercise \ref{ex9b}] It is easy to see that each of the classes listed is the class of
a prime divisor $D$ with $D^2<0$; for example, $L-E_1-E_2-E_3$ is the class of the proper transform 
of the line tangent to $C$ at $p_1$ (i.e., the flex line), while the 
the class of the proper transform of 
$C$ to $X$ is $3L-E_1-\cdots-E_r=-K_X$, which has negative self-intersection exactly when $r>9$.
Suppose $D=aL-a_1E_1-\cdots-a_rE_r$ is the class of a prime divisor with 
$D^2<0$ which is not in the given list.
Then $D$ meets each of the listed classes non-negatively; i.e.,
$D\cdot E_r\ge 0$ (so $a_r\ge0$), $D\cdot (E_i-E_{i+1})\ge0$ (so
$a_i\ge a_{i+1}$ for each $i=1,\ldots,r-1$, hence $a_1\ge a_2\ge \cdots\ge a_r\ge 0$) 
and $D\cdot (L-E_1-E_2-E_3)\ge0$
(so $a\ge a_1+a_2+a_3$). It is not hard to see the non-negative integer linear
combinations of $L$, $L-E_1$, $2L-E_1-E_2$, $3L-E_1-E_2-E_3$, $\ldots$,
$3L-E_1-\cdots-E_r=-K_X$ are precisely the classes
which meet $E_r$, $E_i-E_{i+1}$ for $i>0$ and $L-E_1-E_2-E_3$ non-negatively
(see \cite{refHa6}). But each of $L$, $L-E_1$, $2L-E_1-E_2$, $3L-E_1-E_2-E_3$, $\ldots$,
$3L-E_1-\cdots-E_r=-K_X$ is
a sum of the listed classes of negative self-intersection (for example,
$L=(L-E_1-E_2-E_3)+
((E_1-E_2)+\cdots+(E_r-E_{r-1})+E_r)+((E_2-E_3)+\cdots+(E_r-E_{r-1})+E_r)
+((E_3-E_4)+\cdots+(E_r-E_{r-1})+E_r)$ and
$3L-E_1-\cdots-E_{r-1}=-K_X+E_r$; moreover, if $r=9$, then
$-K_r=3(L-E_1-E_2-E_3)+2(E_1-E_2)+4(E_2-E_3)+6(E_3-E_4)+5(E_4-E_5)
+4(E_5-E_6)+\cdots+(E_8-E_9)$, if $r=8$, then
$-K_r=3(L-E_1-E_2-E_3)+2(E_1-E_2)+4(E_2-E_3)+6(E_3-E_4)+5(E_4-E_5)
+4(E_5-E_6)+\cdots+2(E_7-E_8)+E_8$, etc.). Thus 
$D\cdot(-K_X)\ge0$ implies $D$ is a sum of classes of negative self-intersection, 
each of which it meets non-negatively,
so $D^2\ge0$. Thus our list of classes of prime divisors of negative self-intersection is complete.
Because $F$ meets $E_r$, $E_i-E_{i+1}$ for $i>0$ and $L-E_1-E_2-E_3$ non-negatively
if $F$ is nef, this also shows that 
any nef class $F$ is a non-negative integer linear
combination of $L$, $L-E_1$, $2L-E_1-E_2$, $3L-E_1-E_2-E_3$, $\ldots$,
$3L-E_1-\cdots-E_r=-K_X$. 

Since each of 
$L$, $L-E_1$, $2L-E_1-E_2$, $3L-E_1-E_2-E_3$, $\ldots$,
$3L-E_1-\cdots-E_r=-K_X$ is a non-negative integer linear
combination of the listed classes of negative self-intersection, we see
the latter generate $\EFF(X)$, and we also see
that a class $F$ is nef if and only if it is a non-negative integer linear
combination $F$ of $L$, $L-E_1$, $2L-E_1-E_2$, $3L-E_1-E_2-E_3$, $\ldots$,
$3L-E_1-\cdots-E_r=-K_X$ with $-K_X\cdot F\ge0$.
\end{Soln}

\newpage

\section{Lecture: Abnormality}

\subsection{Abnormal Curves}\label{abn}

One of the difficulties in studying Conjecture \ref{folkloreconj} and Problem \ref{asympprob}
is the possibility of there being infinitely many prime divisors $C$ with $C^2<0$, possibly 
(for all anyone knows) even with $C^2$ arbitrarily
negative. As an intermediate step, it might be worthwhile to define and study a class of effective
divisors $C$ with $C^2<0$ which are so bad as to form a finite set. Doing so turns out to have useful 
applications to computing $\varepsilon(\pr2;p_1,\ldots,p_r)$. 

\begin{Def} Consider a surface $X$ obtained by blowing up a finite set of points $p_1,\ldots,p_r\in\pr2$.
Let $C=dL-\sum_im_iE_i\in\EFF(X)$ and assume $m_i\ge0$ for all $i$ with $m_i>0$ for some $i$.
Working formally (i.e., in $\Cl(X)\otimes_{\bf Z}{\bf Q}$), 
let ${\overline C}=dL-{\overline m}\sum_iE_i$, where ${\overline m}=(\sum_im_i)/r$.
Following Nagata \cite{refN}, we say $C$ is {\it abnormal} if ${\overline C}^2<0$.
This is equivalent to $d/\sum_im_i<1/\sqrt{r}$, and also to $d/(r{\overline m})<1/\sqrt{r}$. 
\end{Def}

We note that not every curve $C$ with $C^2<0$ is abnormal (see Exercise \ref{ex10}); in fact,
$X$ has at most finitely many prime divisors which are abnormal curves (see Exercise \ref{ex11}),
but $X$ can have infinitely many prime $C$ with $C^2<0$ (see Exercise \ref{ex6}).

One application of the concept of abnormality is to computing $\varepsilon(\pr2;p_1,\ldots,p_r)$:

\begin{thm}\label{abnormalscomputethm}
Let $X$ be a surface obtained by blowing up a finite set of points $p_1,\ldots,p_r\in\pr2$.
Then $\varepsilon(\pr2;p_1,\ldots,p_r)<1/\sqrt{r}$ if and only if $X$ has an abnormal prime divisor.
\end{thm}

\begin{proof} If $X$ has an abnormal prime divisor, then $\varepsilon(\pr2;p_1,\ldots,p_r)<1/\sqrt{r}$ 
follows by definition of $\varepsilon$. Conversely, assume 
$\varepsilon(\pr2;p_1,\ldots,p_r)<1/\sqrt{r}$. Then there is a class $C=dL-\sum_im_iE_i\in\EFF(X)$
with $d/\sum_im_i<1/\sqrt{r}$ and hence ${\overline C}^2<0$. 
Write $C=\sum_iC_i$ as a sum of prime divisors $C_i$. 
We may assume no summand is of the form $E_j$, since after removing all such summands
we still have an abnormal curve. Thus every summand $C_j$ is of the form $d_jL-\sum_im_{ij}E_i$
with $m_{ij}\ge 0$. 
Now ${\overline C}=\sum_i{\overline C}_i$, so ${\overline C}^2=\sum_{ij}{\overline C}_i\cdot {\overline C}_j<0$
hence ${\overline C}_i\cdot {\overline C}_j<0$
for some $i$ and $j$. But if neither $C_i$ nor $C_j$ were abnormal, then it is easy to see that
${\overline C}_i\cdot {\overline C}_j\ge0$.
\end{proof}

\begin{cor}\label{abnormalscomputecor}
Let $X$ be a surface obtained by blowing up a finite set of points $p_1,\ldots,p_r\in\pr2$.
If there are no prime divisors on $X$ which are abnormal, then $\varepsilon(\pr2;p_1,\ldots,p_r)=1/\sqrt{r}$.
If there are abnormal prime divisors on $X$, then $\varepsilon(\pr2;p_1,\ldots,p_r)=d/\sum_im_i<1/\sqrt{r}$ for some 
abnormal prime divisor $C=dL-\sum_im_iE_i$.
\end{cor}

\begin{proof} 
The first statement follows from Theorem \ref{abnormalscomputethm}. The second
follows from the fact that there are only finitely many
abnormal prime divisors and hence the infimum in the definition of
$\varepsilon(\pr2;p_1,\ldots,p_r)$ as given in (\ref{altdefvareps}) is actually a minimum
(see Exercises \ref{ex7}(a) and \ref{ex11}). The fact that $d/\sum_im_i<1/\sqrt{r}$
is just the definition of abnormality.
\end{proof}

The values of $\varepsilon(\pr2;p_1,\ldots,p_r)$ are known when $X$ is obtained by blowing up $r$
generic points of $\pr2$ if either $r\le 9$ or $r$ is a square. 
It is an open problem to compute $\varepsilon(\pr2;p_1,\ldots,p_r)$  when $r>9$ is not a square.
There is a long-standing conjecture, however, which implies 
(and in fact is equivalent to) $\varepsilon(\pr2;p_1,\ldots,p_r)=1/\sqrt{r}$ for $r>9$:

\begin{Conj}[Nagata \cite{refN2}]\label{NagataConj} If $X$ is obtained by blowing up $r>9$
generic points of $\pr2$, then $X$ has no abnormal curves.
\end{Conj}

Nagata proved this when $r$ is a square \cite{refN2}. 
The conjecture is still open, although
it is known in various special cases.
For example, the conjecture is equivalent to: 

\begin{Conj}\label{NagataConjVersion2} If $dL-m(E_1+\cdots+E_r)\in\EFF(X)$
when $X$ is obtained by blowing up $r>9$ generic points of $\pr2$, then $d>m\sqrt{r}$.
\end{Conj} 

By \cite[Corollary 4.1]{refHR3}, this is true when $m\le t(t-3)/2$, where $t=\lfloor\sqrt{r}\rfloor$.
In addition, Dumnicki shows Conjecture \ref{NagataConjVersion2} is true when $m\le 42$.

\subsection{A Dual Problem}\label{dual}

Let $X$ be obtained by blowing up $r$
points of $\pr2$.
Recall that $\varepsilon(\pr2;p_1,\ldots,p_r)$ is the supremum of $1/t$ over all $t$ such that
$tL-\sum_iE_i\in\NEF(X)$. There is a dual notion which Chudnovsky \cite{refCh} attributes
to Waldschmidt \cite{refW}. 

\begin{Def}\label{gammadef}
$$\gamma(\pr2;p_1,\ldots,p_r)=\inf\Big\{ \frac{d}{m}:dL-m\sum_iE_i\in\EFF(X) \Big\}$$
\end{Def}

Clearly (as Chudnovsky \cite{refCh} remarks), we have
$r\varepsilon(\pr2;p_1,\ldots,p_r)\le \gamma(\pr2;p_1,\ldots,p_r)$.
Also, since $d/m>\sqrt{r}$ implies $(dL-m\sum_iE_i)^2>0$, we see 
in that case for $D=dL-m\sum_iE_i$
that $sD\in\EFF(X)$ for $s\gg0$, hence $\gamma(\pr2;p_1,\ldots,p_r)\le d/m$ for all $d/m>\sqrt{r}$;
i.e., $\gamma(\pr2;p_1,\ldots,p_r)\le \sqrt{r}$. Thus we have:

\begin{cor}\label{epsgammacor}
$$r\varepsilon(\pr2;p_1,\ldots,p_r)\le \gamma(\pr2;p_1,\ldots,p_r)\le \sqrt{r}$$
\end{cor}
As Chudnovsky \cite{refCh} points out, 
although in general $r\varepsilon(\pr2;p_1,\ldots,p_r) < \gamma(\pr2;p_1,\ldots,p_r)$ (see Exercise \ref{ex13}),
if the points $p_1,\ldots,p_r$ are generic we have $r\varepsilon(\pr2;p_1,\ldots,p_r) = \gamma(\pr2;p_1,\ldots,p_r)$ (Exercise \ref{ex12}).
Thus Nagata's conjecture (Conjecture \ref{NagataConj}) is also equivalent to
$\gamma(\pr2;p_1,\ldots,p_r)=\sqrt{r}$ for $r>9$ generic points of $\pr2$.
Chudnovsky \cite{refCh} also remarks that $\gamma(\pr2;p_1,\ldots,p_r)$ is actually a limit:

\begin{prop}\label{gammalimit} Let $X$ be obtained by blowing up $r$ distinct points $p_i\in\pr2$. Then
$$\gamma(\pr2;p_1,\ldots,p_r)=\lim_{m\to\infty}\frac{d_m}{m}$$
where $d_m$ is the least $t$ such that $tL-m\sum_iE_i\in\EFF(X)$. Moreover, for each $n\ge1$,
we have 
$$\gamma(\pr2;p_1,\ldots,p_r)\le \frac{d_n}{n}.$$
\end{prop}

\begin{proof}
Clearly, $d_{rm}\le rd_m$, so $d_{rm}/(rm)\le d_m/m$. Therefore, $d_{m!}/(m!)\le d_n/n$ for every $n|m!$. 
Thus $d_{m!}/(m!)$ is a non-increasing sequence, so $\lim_{m\to\infty}\frac{d_{m!}}{m!}$ exists;
call it $l$. Given any $\delta>0$, we check for all $n\gg0$ that $l\le d_n/n\le l+\delta$.
Pick $m$ large enough that $l\le \frac{d_{m!}}{m!}\le l+\delta/2$.
Say $n\ge m!$ and write $n=a(m!)+c$, where $c$ is an integer with $0\le c<m!$.
Then $d_n\le d_{(a+1)(m!)}\le (a+1)d_{m!}$ so 
$$l\le \frac{d_{n!}}{n!}\le \frac{d_n}{n}\le \frac{(a+1)d_{m!}}{a(m!)+c} = 
\frac{ad_{m!}}{a(m!)+c} + \frac{d_{m!}}{a(m!)+c}\le
\frac{d_{m!}}{m!} + \frac{d_{m!}}{a(m!)}\le l+\delta/2+d_{m!}/(a(m!)),$$
and for $n\gg0$ we will have $a$ large enough such that $d_{m!}/(a(m!))\le \delta/2$.
We also now see $\gamma(\pr2;p_1,\ldots,p_r)\le \frac{d_{n!}}{n!}\le \frac{d_n}{n}$.
\end{proof}

\vskip.5in

\subsection{Exercises}

\begin{Ex}\label{ex10} Let $X$ be obtained by blowing up $r$ points $p_i\in\pr2$.
\begin{itemize}
\item[(a)] Show that $C^2\le {\overline C}^2$ for any divisor $C$ on $X$.
\item[(b)] If $C_1=a_0L-\sum_ia_i$ and $C_2=b_0L-\sum_ib_i$,
where $a_1\ge a_2\ge \cdots\ge a_r\ge0$ and $b_1\ge b_2\ge \cdots\ge b_r\ge0$,
show that $C_1\cdot C_2\le {\overline C}_1\cdot {\overline C}_2$.
\item[(c)] If $C_1$ and $C_2$ are abnormal, show that ${\overline C_1}\cdot {\overline C_2}<0$.
\item[(d)] Give an example of a curve $C$ with $C^2<0$ but such that $C$ is not abnormal.
\end{itemize}
\end{Ex}

\begin{Soln}[Exercise \ref{ex10}] (a) Let $C=dL-\sum_{i=1}^rm_iE_i$.
Thus we need to show that $d^2-\sum_im_i^2\le d^2-r{\overline m}^2$;
i.e., that $\sum_im_i^2\ge r{\overline m}^2$. Let $v=(m_1,\ldots,m_r)$
and let ${\overline v}=({\overline m},\ldots,{\overline m})$.
Then we need to show, with respect to the Euclidean dot product,
that $0\le v^2-{\overline v}^2$, but ${\overline v}^2={\overline v}\cdot v$, so 
$0\le(v-{\overline v})^2=v^2+{\overline v}^2-2v\cdot{\overline v}
=v^2-{\overline v}^2$, as required.

(b) If $a=\sum_ia_i/r$ and $b=\sum_ib_i/r$, it suffices to show that 
$\sum_ia_ib_i\ge rab$. But $rab=a\sum_ib_i$, so we need only show
$\sum_ia_ib_i\ge a\sum_ib_i$. This is equivalent to showing
$\sum_i(ra_i)b_i\ge (ra)\sum_ib_i$, where $ra=\sum_ia_i$;
i.e., we can reduce to the case that $a$ is an integer. 
If $a_1,\cdots,a_r$ are not all equal, we can
pick some $j$ such that $a_j>a$ and some $l$ such that $a>a_l$.
Let $a'_j=a_j-1$ and $a'_l=a_l+1$, and $a'_i=a_i$ for $i\ne j,l$.
Then $\sum_i(a_i-a)^2 >\sum_i(a'_i-a)^2$ and $\sum_ia_ib_i=
(b_j-b_l)+\sum_ia'_ib_i\ge \sum_ia'_ib_i$. By repeating this procedure
we eventually obtain a sequence $a'_i$, $1\le i\le r$, such that $\sum_i(a'_i-a)^2=0$ and hence
$a=a'_i$ for all $1\le i\le r$ and so $\sum_ia_ib_i\ge \sum_ia'_ib_i=rab$.

(c) Let ${\overline C_1}=aL-b\sum_iE_i$ and let ${\overline C_2}=cL-d\sum_iE_i$.
Then $a/(rb)<1/\sqrt{r}$ and $c/(rd)<1/\sqrt{r}$, so $ac/(rbd)<1$
so ${\overline C_1}\cdot {\overline C_2}<0$.

(d) An easy example is given by $L-E_1-E_1-0E_3-0E_4$.
For a more interesting example, choose an irreducible quartic plane curve $C'$
with a triple point. Blow up the triple point and eight additional points on $C'$.
The proper transform of $C'$ is $C=4L-3E_1-E_2-\cdots-E_9$.
Then ${\overline C}^2>0$ but $C^2=-1$.
More generally, if you blow up 9 or more general enough points of $\pr2$, then
there are infinitely many exceptional curves (i.e., the prime divisors $E$ 
with $E^2=E\cdot K_X=-1$) by Exercise \ref{ex6}(c), but by Exercise \ref{ex11}
at most finitely many of them are abnormal.
\end{Soln}

\begin{Ex}\label{ex11} Let $X$ be obtained by blowing up $r$ points $p_i\in\pr2$.
Then there are at most finitely many prime divisors $C$ which are abnormal.
In fact, there are at most $r+1$ of them \cite{refS}.
\end{Ex}

\begin{Soln}[Exercise \ref{ex11}] Suppose there were an infinite set $S$ of them.
We get a mapping $\phi:S\to S_r$ by choosing, for each $C\in S$, a permutation
$\pi$ such that if $C=dL-\sum_im_iE_i$, then
$m_{\pi(1)}\ge m_{\pi(2)}\ge \cdots\ge m_{\pi(r)}$.
Thus there must be two prime divisors $C_1\ne C_2$ with $\phi(C_1)=\phi(C_2)$ if $S$ is infinite.
Hence by Exercise \ref{ex10}(b, c) we have $C_1\cdot C_2<0$, but $C_1\ne C_2$ implies
$0\le C_1\cdot C_2$, which is impossible.

To see that there are at most $r+1$, suppose there were more, say $C_i$,
for $1\le i\le t$ for $t>r+1$. Since $\Cl(X)$ has rank
$r+1$, there is a relation $\sum_im_iC_i=0$ where the $C_i$ are distinct. 
Let $P=\sum_{i, m_i>0}m_iC_i$ and let $N=-\sum_{i, m_i<0}m_iC_i$. Then $P-N=0$
hence $P=N$. Now, $P$ is abnormal, hence $P^2<0$, but $P^2=P\cdot N\ge 0$
which is a contradiction.
\end{Soln}

\begin{Ex}\label{ex13} Give an example such that 
$r\varepsilon(\pr2;p_1,\ldots,p_r)<\gamma(\pr2;p_1,\ldots,p_r)$.
\end{Ex}

\begin{Soln}[Exercise \ref{ex13}] 
Consider four points $p_1,\ldots,p_4$, exactly three of which 
(say $p_1,p_2,p_3$) are collinear. Let $C=L-E_1-E_2-E_3$ and let
$L_{ij}=L-E_i-E_j$.
Then $F=3L-E_1-E_2-E_3-E_4=C+L+(L-E_4)$ and $H=3L-E_1-E_2-E_3-2E_4=C+2(L-E_4)$
are nef and $C=5L-3E_1-3E_2-3E_3-3E_4=2C+L_{14}+L_{24}+L_{34}\in\EFF(X)$.
Since $H\cdot C=0$, we see $\gamma(\pr2;p_1,\ldots,p_r)=5/3$, and since $F\cdot C=0$,
we see by Exercise \ref{ex7b} that $\varepsilon(\pr2;p_1,\ldots,p_r)=1/3$.
\end{Soln}

\begin{Ex}\label{ex12} Let $X$ be obtained by blowing up $r$ generic points $p_i\in\pr2$.
\begin{itemize}
\item[(a)] Compute $\varepsilon(\pr2;p_1,\ldots,p_r)$ for each $r\le 9$ and each $r$ which is a perfect square.
\item[(b)] Show that $r\varepsilon(\pr2;p_1,\ldots,p_r)=\gamma(\pr2;p_1,\ldots,p_r)$.
\end{itemize}
\end{Ex}

\begin{Soln}[Exercise \ref{ex12}] (a) Suppose $r=d^2$ is a perfect square.
Let $X'$ be obtained by blowing up $r=d^2$ points $p_i'\in\pr2$ on a smooth 
plane curve $C$ of degree $d$, hence the class $C'=dL'-\sum_iE'_i$ of the proper transform of $C$ is
nef. Let $X$ be obtained by blowing up $r=d^2$ generic points $p_i\in\pr2$. 
Since by Theorem \ref{SCprin2} for any divisor $tL'-\sum_iE_i'\in\NEF(X')$ we also
have $tL-\sum_iE_i\in\NEF(X)$ we see that $\varepsilon(\pr2;p'_1,\ldots,p'_r)\le\varepsilon(\pr2;p_1,\ldots,p_r)$.
But $C'\in\NEF(X')\cap\EFF(X')$ together with $(C')^2=0$ implies that $\varepsilon(\pr2;p'_1,\ldots,p'_r)=1/d$
by Remark \ref{easySCcomp}. On the other hand, $\varepsilon(\pr2;p_1,\ldots,p_r)\le 1/\sqrt{r}$ by 
Corollary \ref{abnormalscomputecor}. 
Thus $\varepsilon(\pr2;p_1,\ldots,p_r)=1/\sqrt{r}$ when $r$ is a perfect square.

For the case of $r\le 9$ generic points, if $C=tL-\sum_im_iE_i$ is the class of an abnormal prime divisor,
then so is $C'=tL-\sum_im'_iE_i$, where the $m'_i$ are obtained by a permutation of the $m_i$
such that $m'_1\ge m'_2\ge \cdots\ge m'_r$. Given two abnormal prime divisors, $C$ and $D$, we thus see
that the permuted divisors $C'$ and $D'$ have $C'\cdot D'<0$, and hence $C'=D'$. I.e., up to permutations,
$C$ and $D$ are the same, so if $X$ has any abnormal curve, that curve gives the value 
of $\varepsilon(\pr2;p_1,\ldots,p_r)$. 

Thus, since $C=L-E_1-E_2$ is abnormal for $r=2$ or 3, we see
$\varepsilon(\pr2;p_1,p_2)=\varepsilon(\pr2;p_1,p_2,p_3)=1/2$. 
For $r=5$ or 6, take $C=2L-E_1-\cdots-E_5$
to see $\varepsilon(\pr2;p_1,\ldots,p_5)=2/5$. 
For $r=7$, take $C=3L-2E_1-E_2-\cdots-E_7$ to see that $\varepsilon(\pr2;p_1,\ldots,p_7)=3/8$,
and for $r=8$, take $C=6L-3E_1-2E_2-\cdots-2E_8$ to see that $\varepsilon(\pr2;p_1,\ldots,p_7)=6/17$.
(For the fact that $3L-2E_1-E_2-\cdots-E_7$ and $6L-3E_1-2E_2-\cdots-2E_8$ are classes of prime divisors,
use Exercise \ref{ex6} over the complex numbers. More generally, one can use quadratic transforms
to see that $3L-2E_1-E_2-\cdots-E_7$ and $6L-3E_1-2E_2-\cdots-2E_8$ are smooth rational curves.)

(b) Since the points are generic, if $C=tL-\sum_im_iE_i$ is the class of an effective divisor,
then so is $C'=tL-\sum_im'_iE_i$, where the $m'_i$ are obtained by any permutation of the $m_i$.
Thus $r{\overline C}\in\EFF(X)$ for any $C=aL-a_1E_1-\cdots-a_rE_r\in\EFF(X)$. 
But $r{\overline C}=raL-(a_1+\cdots+a_r)(E_1+\cdots+E_r)$
and $ar/(r(a_1+\cdots+a_r))=a/(a_1+\cdots+a_r)$, so 
$r\varepsilon(\pr2;p_1,\ldots,p_r)\ge\gamma(\pr2;p_1,\ldots,p_r)$. This together with
Corollary \ref{epsgammacor} gives 
$r\varepsilon(\pr2;p_1,\ldots,p_r)=\gamma(\pr2;p_1,\ldots,p_r)$.
\end{Soln}

\newpage

\section{Lecture: Computation of Seshadri Constants}

\subsection{Estimating Seshadri Constants}\label{computingSCs}

Given distinct points $p_i\in\pr2$ we now consider the problem of 
estimating $\varepsilon(\pr2;p_1,\ldots,p_r)$. Getting an upper bound
less than $1/\sqrt{r}$ is, by Corollary \ref{abnormalscomputecor}, 
equivalent to showing the existence of abnormal curves, and this is often quite hard.
Thus much of the focus has been on getting increasingly better lower bounds.

There have been two main methods used for this. Both methods can be adapted to
studying Seshadri constants on surfaces in general. For purposes of exposition
we will continue to focus on the case of $\pr2$.

The first method is to explicitly construct
nef divisors. For example, if one shows some divisor $F=dL-m\sum_iE_i$
is nef, then we know $m/d\le \varepsilon(\pr2;p_1,\ldots,p_r)$.  This is the method used by
\cite{refBi}, \cite{refHa5} and \cite{refHa4}. Both authors first construct
a nef divisor $F'=d'L-\sum_im_iE_i$, and then use an averaging process
to get a nef divisor of the form $F=dL-m\sum_iE_i$.

The second main method is to rule out the possible occurrence of abnormal curves.
This method has been applied by \cite{refX}, \cite{refT}, \cite{refST}, \cite{refSS},
\cite{refHR1} and \cite{refHR2}. Fundamentally it depends on the fact that if
$F=tL-m\sum_iE_i$ has $F^2>0$, then, as we show below, there are only finitely many classes
$C=dL-\sum_im_iE_i$ that could possibly be the class of a prime divisor
with $F\cdot C<0$ \cite[Lemma 2.1.3]{refHR1}. If one can show that none of these finitely many
classes is the class of a prime divisor, then $F$ is nef and $m/t\le 
\varepsilon(\pr2;p_1,\ldots,p_r)$. 

\begin{prop}\label{FabnormalFinite} Let $X$ be obtained by blowing up
distinct points $p_1,\ldots,p_r\in\pr2$, with $L$ and $E_i$ as usual. 
Assume that $F=tL-m\sum_iE_i$ has $F^2>0$ and $t>0$.
Then there is an explicitly computable finite set $S_F$ of classes 
which contains the class of every prime divisor $C$ with $C\cdot F<0$ (if any).
\end{prop}

\begin{proof} Since $F^2>0$ and $F\cdot L>0$, we can find an 
explicit $s$ such that $sF\in\EFF(X)$ (but the smaller $s$ is the smaller 
$S_F$ will be). 

Let $E=\sum_iE_i$, and choose 
nef divisors $H_i$ that span $\Cl(X)$. For example,
$H_0=L$, and $H_i=L-E_i$ for $i>0$, or choose
$h_i>0$ large enough such that
$H_0=h_0L-E$ and $H_i=h_iL-E-E_i$ are in $\EFF(X)\cap \NEF(X)$.
(It is clear that $h_0=r$ and $h_i=r+1$ will suffice, but the smaller
one can choose the $h_i$ the smaller $S_F$ will be. Being able to choose
smaller values of the $h_i$ will depend on having some knowledge of how the
points $p_i$ are arranged, since if the points are collinear, then 
$h_0=r$ and $h_i=r+1$ are best possible.)

If $C$ is the class of a prime divisor with $F\cdot C<0$, then
$sF-C$ and $C$ are both in $\EFF(X)$, hence both meet every $H_i$
non-negatively so $0\le C\cdot H_i \le sF\cdot H_i$ for each $i$. 

Since the classes $H_i$ generate $\Cl(X)$, if for two classes $C_1$ and $C_2$ 
we have $C_1\cdot H_i=C_2\cdot H_i$ for all $i$, then $C_1=C_2$.
Thus there are only finitely many possible classes $C$ with 
$0\le C\cdot H_i \le sF\cdot H_i$ for all $i$. 
\end{proof}

When the $r$ points $p_i$ are general, one can narrow down the set $S_F$ even more \cite{refS}.

\begin{lem}\label{almostunif} Let $X$ be obtained by blowing up
general points $p_1,\ldots,p_r\in\pr2$, with $L$ and $E_i$ as usual. 
Assume that $C=dL-\sum_im_iE_i$ is abnormal; then all but at most one of the
coefficients $m_i$ are equal.
\end{lem}

\begin{proof} 
By Exercise \ref{ex11} there are at most $r+1$ prime divisors $C=dL-\sum_im_iE_i$
which are abnormal, but since the points are general
any permutation of the $m_i$ is again an abnormal prime divisor. 
We may assume that $m_1\ge \cdots\ge m_r\ge0$.
Suppose that there is an index $i$ such that $m_1>m_i>m_r$.
Then there are $i-2$ permutations $\omega_j$ which are transpositions of $m_r$
with $m_j$, where $1<j< i$. There are $r-i-1$ more transpositions $\alpha_j$
of $m_1$ with $m_j$, where $i< j<r$. In addition, there are six permutations
in which we permute $m_1$, $m_i$ and $m_r$ with each other only.
This gives $(i-2)+(r-i-1)+6=r+3$ distinct permutations, contradicting there being
at most $r+1$ abnormal prime divisors. Thus at most two values can occur among the $m_i$.
The only other possibility to be ruled out is if the two values each occur at least twice.
So assume that
$m_1=\cdots=m_j>m_{j+1}=\cdots=m_r$, where $r\ge 4$ and $2\le j\le r-2$. The number of
distinct arrangements of the $m_i$ is $\binom{r}{j}$. Looking at Pascal's triangle it is clear that
$\binom{r}{j}>r+1$. (Since the entries in the triangle we're interested in are on the row 
beginning $1\ r\ \cdots$, but more than two spots from either end, we see 
$\binom{r}{j}$ is the sum of two entries on the row above it, each entry being at least $r-1$,
so $\binom{r}{j}\ge2r-2$, hence $2r-2 > r+1$, since $r\ge 4$.)
\end{proof}

The restrictions on possible abnormal prime divisors can be made even more stringent; see
\cite{refHR1} and \cite{refHR2}. 

\begin{Eg}\label{almunifeg}
Suppose we blow up six general points $p_1,\ldots,p_6$.
We will use the method of ruling out abnormal curves
to check that $F=5L-2\sum_iE_i$ is nef, and hence
that $\varepsilon(\pr2;p_1,\ldots,p_6)\ge 2/5$. Since $C=2L-E_1-\cdots-E_5\in\EFF(X)$
has $F\cdot C=0$, this shows $\varepsilon(\pr2;p_1,\ldots,p_6)=2/5$ by Exercise \ref{ex7b}.

First, note that $H_0=3L-\sum_iE_i=(2L-E_1-\cdots-E_4)+(L-E_5-E_6)$ is nef, since
each summand is the class of a prime divisor which $H_0$ meets non-negatively.
Also, $H_i=H_0-E_i\in\NEF(X)$. For example, $H_5=(2L-E_1-\cdots-E_5)+(L-E_5-E_6)$,
but $H_5$ meets each summand non-negatively, each of which is the class of a prime divisor.

Suppose $C$ is the class of a prime divisor such that $0>C\cdot F$.
Then $C$ is abnormal and by Lemma \ref{almostunif}
we may (after reindexing, if need be) assume that $C=dL-m\sum_iE_i-kE_1$ for some $k$.
First suppose $k=0$. Then we have $0\le C\cdot H_0=3d-6m\le F\cdot H_0=3$,
$0\le C\cdot H_1=3d-7m\le F\cdot H_1=1$.

Thus $7m\le 3d\le 6m+3$, so $m\le 3$. For $m=1$ we get $d=3$ (which fails $5d-12m=F\cdot C<0$), 
for $m=2$ we get $d=5$ (which also fails $5d-12m<0$),
and for $m=3$ we get $d=7$. But since $C$ is supposed to be a prime divisor it should
satisfy adjunction and thus must have $-2\le C^2 +C\cdot K_X$, but 
for $d=7$ with $m=3$ we find $C^2 +C\cdot K_X=-8$.

So suppose $C=dL-m\sum_iE_i-kE_1$ for some $k>0$ so $5d-12m-2k=C\cdot F<0$. 
We have $0\le C\cdot H_0=3d-6m-k\le F\cdot H_0=3$ and
$0\le C\cdot H_1=3d-7m-2k\le F\cdot H_1=1$.
Thus $7m+2k\le 3d\le 6m+k+3$, so $m+k\le 3$. Thus $(d,m,k)$ is either
$(1,0,1)$, $(2,0,3)$ or $(3,1,1)$, giving $C=L-E_1$ (which fails $F\cdot C<0$),
$C=2L-3E_1$ (which is not in $\EFF(X)$), and $C=3L-2E_1-E_2-\cdots-E_6$
(which also fails $F\cdot C<0$). 

Finally, assume $C=dL-m\sum_iE_i-kE_1$ for some $k<0$, so $5d-12m-2k=C\cdot F<0$.
Since $C$ should be the class of a prime divisor with $C\cdot L>0$, we have 
$C\cdot E_i\ge0$ for all $i$, hence $-k\le m$. 
We have $0\le C\cdot H_0=3d-6m-k\le F\cdot H_0=3$ and
$0\le C\cdot H_1=3d-7m-2k\le F\cdot H_1=1$.
Thus $7m+2k\le 3d\le 6m+k+3$ and hence also 
$5m\le7m+2k\le 3d$. Since $F=(2-E_1-\cdots-E_5)+(2-E_2-\cdots-E_6)+(L-E_1-E_6)$, 
we see $F\in\EFF(X)$, hence $F-C$ should also be effective, so
$d\le 5$, whence $5m\le 3d\le 15$ implies $1\le-k\le m\le 3$. 
The simultaneous solutions to $7m+2k\le 3d\le 6m+k+3$, $1\le-k\le m\le 3$ and
 $5d-12m-2k<0$ are $(d,m,k)\in\{(5,3,-3), (6,3,-2), (4,2,-1)\}$. None of these are effective.
For example, $E=2L-E_2-\cdots-E_6$ is a prime divisor, but $C=5L-3(E_2+\cdots+E_6)$
for $(d,m,k)=(5,3,-3)$; since $E\cdot C<0$, $C-E$ is effective if $C$ is, and likewise so are
$C-2E$ and $C-3E$, but $C-3E=-L$ is not effective, hence neither is $C$. The same argument
handles the other two cases.
 
Thus $F$ is nef, as claimed.
\end{Eg}

We now give an example of the alternative approach using the method of \cite{refHa5}
and \cite{refHa4}, based on the idea of unloading \cite{refR}.

\begin{prop}\label{altapp} Let $d,r,n$ be positive integers such that $r < d\sqrt{n}$ and $r\le n$.
Then for $n$ general points $p_i$,
we have $$\varepsilon(\pr2;p_1,\ldots,p_n)\ge \frac{r}{nd}.$$
\end{prop}

\begin{proof} 
It is enough to show that $ndL-r(E_1+\cdots+E_n)\in\NEF(X)$, where $X$ is the blow up of
$\pr2$ at general points $p_1,\ldots,p_n$. 
By Theorem \ref{SCprin}(b), it is enough to find essentially distinct points $p_i'$ of
$X_1=\pr2$ such that $ndL'-r(E_1'+\cdots+E_n')\in\NEF(X'_{n+1})$. Choose any smooth plane curve
$C_1$ of degree $d$. Let $p_1'\in C_1$. Recursively, let $X'_{i+1}$ be the blow up of $X'_i$ at
$p_i'$, let $C_{i+1}$ be the proper transform of $C_i$, and let $p_{i+1}'$ be the point
of $C_{i+1}$ infinitely near to $p_i'$. This defines $p_1',\ldots,p_r'$. If $n>r$, for $r<i\le n$, choose
$p_i'$ to be infinitely near to $p_{i-1}'$ but choose $p_{r+1}'$ not to be on $C_{r+1}$.

Thus $dL'-E_1'-\cdots-E_r'$ is the class of $C_{r+1}$, i.e., the proper transform of $C_1$, hence
the class of a prime divisor, as are $E_i'-E_{i+1}'$ for each $1\le i<n$ and $E_n'$.
In particular $E_i'-E_j'\in\EFF(X'_{n+1})$ for every $j>i$, and hence so is $E_1'+\cdots+E_r'-rE_i'$ for every
$i>r$. Since $ndL'-n(E_1'+\cdots+E_r')\in\EFF(X'_{n+1})$, we see
$$F=ndL'-r(E_1'+\cdots+E_n')=
(ndL'-n(E_1'+\cdots+E_r'))+\sum_{r<i\le n}(E_1'+\cdots+E_r'-rE_i')\in\EFF(X'_{n+1})$$
(this is the unloading step).
The irreducible components of this sum are $C_{r+1}$ and $E_i'-E_{i+1}'$ for various $i$,
but $F$ meets each one non-negatively (this is clear for $E_i'-E_{i+1}'$, and $F\cdot C_{r+1}\ge0$
since $nd^2-r^2>0$). So $F\in\NEF(X'_{n+1})$, as required.
\end{proof}

As another variation we have:

\begin{prop}\label{altapp2} Let $d,r,n$ be positive integers such that $n\ge r> d\sqrt{n}$.
Then for $n$ general points $p_i\in\pr2$,
we have $$\varepsilon(\pr2;p_1,\ldots,p_n)\ge \frac{d}{r}.$$
\end{prop}

\begin{proof} See Exercise \ref{ex16}.
\end{proof}

\begin{Eg}\label{alteg} Again suppose we blow up six general points $p_1,\ldots,p_6$;
let $X$ be the surface we obtain.
Then $5>2\sqrt{6}$, so by Proposition \ref{altapp2}, we see that
$\varepsilon(\pr2;p_1,\ldots,p_6)\ge \frac{2}{5}$ and hence that
$F=5L-2\sum_iE_i\in\NEF(X)$. Since $C=2L-E_1-\cdots-E_5\in\EFF(X)$ 
and $F\cdot C=0$, we see by Remark \ref{easySCcomp}
that in fact $\varepsilon(\pr2;p_1,\ldots,p_6)=\frac{2}{5}$.
\end{Eg}

\vskip.3in

\subsection{Exercises}

\begin{Ex}\label{ex14} Compute $\varepsilon(\pr2;p_1,\ldots,p_r)$ and $\gamma(\pr2;p_1,\ldots,p_r)$
for every choice of $r<9$ distinct points of $\pr2$.
\end{Ex}

\begin{Soln}[Exercise \ref{ex14}] 
This can be done using the various possibilities (worked out
in \cite{refGHM}) for $\EFF(X)$ where $X$ is the blow up of
$\pr2$ at the $r$ points. Some of these cases are discussed in \cite{refCh}. 
\end{Soln}

\begin{Ex}\label{ex15} Let $X$ be the blow up of 12 general points.
Study whether $F=7L-2(E_1+\cdots+E_{12})$ is nef, using the method of
Example \ref{almunifeg}.
\end{Ex}

\begin{Soln}[Exercise \ref{ex15}] 
It is nef, using Proposition \ref{altapp2} with $r=7$ and $d=2$. However, using the method of
Example \ref{almunifeg} one is left with showing that in none of the following cases is
$C=dL-(m+k)E_1-m(E_2+\cdots+E_{12})$ an abnormal prime divisor:
\begin{verbatim}
      d=  7 m= 2 k=  1 
      d= 10 m= 3 k=  0 
      d=  3 m= 1 k= -1 
\end{verbatim}
This is clear for the last case, since $3L-E_1-\cdots-E_{11}$ is not effective.
The other two cases are harder to eliminate, but it is known that except for a few exceptional cases
which do not occur here that general points of small multiplicity impose independent conditions
on curves of degree $d$, if there are curves of degree $d$ passing through the points with the specified
multiplicities. (How big ``small'' is keeps increasing as more research is done, 
but certainly multiplicity at most 3 is covered by the results; see \cite{refDu}.)
\end{Soln}

\begin{Ex}\label{ex16} Let $d,r,n$ be positive integers such that $n\ge r > d\sqrt{n}$.
Then for $n$ general points $p_i\in\pr2$,
we have $$\varepsilon(\pr2;p_1,\ldots,p_n)\ge \frac{d}{r}.$$
\end{Ex}

\begin{Soln}[Exercise \ref{ex16}] Mimic the proof of Proposition \ref{altapp}.
It is enough by the semicontinuity principle to find essentially distinct points $p_i'$ of
$X_1=\pr2$ such that $rdL'-d^2(E_1'+\cdots+E_n')\in\NEF(X'_{n+1})$. Choose any smooth plane curve
$C_1$ of degree $d$. Let $p_1'\in C_1$. Recursively, let $X'_{i+1}$ be the blow up of $X'_i$ at
$p_i'$, let $C_{i+1}$ be the proper transform of $C_i$, and let $p_{i+1}'$ be the point
of $C_{i+1}$ infinitely near to $p_i'$. This defines $p_1',\ldots,p_r'$. If $n>r$, for $r<i\le n$, choose
$p_i'$ to be infinitely near to $p_{i-1}'$ but choose $p_{r+1}'$ not to be on $C_{r+1}$.

Thus $dL'-E_1'-\cdots-E_r'$ is the class of $C_{r+1}$, i.e., the proper transform of $C_1$, hence
the class of a prime divisor, as are $E_i'-E_{i+1}'$ for each $1\le i<n$ and $E_n'$.
In particular $E_i'-E_j'\in\EFF(X)$ for every $j>i$. 
Since $rdL'-r(E_1'+\cdots+E_r')\in\EFF(X'_{n+1})$ and since $r^2> nd^2$, by adding
to $rdL'-r(E_1'+\cdots+E_r')$
the classes $m_nE_n'$ and $m_{ij}(E_i'-E_j')$ with $i\le r$ and $j>r$
for appropriate choices of $m_{ij}\ge0$ (this is the unloading step), we obtain
$F=rdL'-d^2(E_1'+\cdots+E_n')$ with $F^2>0$. But $F\cdot C_{r+1}=0$, $F\cdot (E_i'-E_j')=0$ and $F\cdot E_r'>0$,
so $F\in\NEF(X'_{n+1})$, as required.
\end{Soln}

\begin{Ex}\label{ex17} Show $F=5L-E_1-\cdots-E_{21}$ is ample, when 
the $E_i$ are obtained by blowing up 21 general points $p_i$ of $\pr2$.
\end{Ex}

\begin{Soln}[Exercise \ref{ex17}] 
By Proposition \ref{altapp}, using $r=9$ and $d=2$, we see that 
$$\varepsilon(\pr2;p_1,\ldots,p_{21})\ge \frac{9}{42}.$$
Thus $D=42L-9(E_1+\cdots+E_{21})\in\NEF(X)$, so
clearly $45L-9(E_1+\cdots+E_{21})=9F$
(and even $43L-9(E_1+\cdots+E_{21})$, for that matter) is ample by the Nakai-Moiseson
criterion \cite{refHr} since $F^2>0$ and $F$ meets every curve positively (any prime divisor
orthogonal to $D$ must meet $D+L$ positively, since the only prime divisors 
orthogonal to $L$ are the $E_i$, which meet $D$ positively). 
\end{Soln}

\newpage

\section{Lecture: The Containment Problem (an application to Commutative Algebra)}

\subsection{Background}

The notions we've discussed above can be applied to questions of commutative algebra,
especially problems involving ideals of {\it fat points}. Let $p_1,\ldots,p_s\in\pr{n}$ be distinct points.
Let $R=k[\pr{n}]=k[x_0,\ldots,x_n]$ be the homogeneous coordinate ring of $\pr{n}$.
Let $I(p_i)\subset R$ be the ideal generated by all forms vanishing at $p_i$.
Given a 0-cycle $Z=m_1p_1+\cdots+m_sp_s$ (i.e., an element in the free abelian group on the points $p_i$)
with $m_i\ge0$ for all $i$, let $I(Z)$ be the homogeneous ideal $\cap_i I(p_i)^{m_i}$.
This is a saturated ideal which defines a 0-dimensional subscheme of $\pr{n}$. We will abuse notation
and use the 0-cycle $Z=m_1p_1+\cdots+m_sp_s$ to denote this subscheme, which we refer to as a 
fat point subscheme. We will denote the sheaf of ideals corresponding to $I(Z)$ by
$\shf{I}_Z$, hence $I(Z)=\oplus_{t\ge0}H^0(\pr{n},\shf{I}_Z(t))$, where 
$\shf{I}_Z(t)=\shf{I}_Z\otimes_{\OO_{\pr{n}}}\!\!\OO_{\pr{n}}(t)$.
In fact, more is true:

\begin{prop}\label{fatptprop} Given distinct points $p_i\in\pr{n}$ and integers $m_i$.
Let $Z$ be the fat point scheme $\sum_{m_i\ge0}m_ip_i$, let $\pi:X\to \pr{n}$ be the 
morphism obtained by blowing up the points $p_i$,
let $H$ be the pullback to $X$ of a general hyperplane and let $E_i$ be the blow up of $p_i$.
Then there is a natural isomorphism $\shf{I}_{Z}(t)\cong\pi_*(\OO_X(tH-\sum_im_iE_i))$ such that
$H^0(\pr{n},\shf{I}_Z(t))\cong H^0(X,\OO_X(tH-\sum_ia_iE_i))$
and so $I(Z)$ can be identified with $\oplus_{t\ge0}H^0(X,\OO_X(tH-\sum_im_iE_i))$. Moreover,
if $m_i\ge0$ for all $i$, then $H^q(\pr{n},\shf{I}_Z(t))\cong H^q(X,\OO_X(tH-\sum_ia_iE_i))$ 
holds for all $q\ge0$.
\end{prop}

\begin{proof} First, $\shf{I}_Z=\Pi_{m_i\ge0}\shf{I}_{p_i}^{m_i}$. If $m\ge0$ and
$\pi$ is the blow up of a single point $p\in \pr{n}$ where we set 
$E=\pi^{-1}(p)$, then we have a natural morphism
$\shf{I}_p^m\to \pi_*(\pi^{-1}\shf{I}_p^m)$ which induces a morphism
$\shf{I}_p^m\to \pi_*(\pi^{-1}\shf{I}_p^m\cdot \OO_X)=\pi_*\OO_X(-mE)$
and thus 
$$\shf{I}_Z=\Pi_{m_i\ge0}\shf{I}_{p_i}^{m_i}\to\Pi_{m_i\ge0}\pi_*\OO_X(-m_iE)
=\pi_*\OO_X(\sum_{m_i\ge0}-m_iE)\hookrightarrow\pi_*\OO_X(\sum_i-m_iE).$$
By the projection formula (\cite[Exercise II.5.1(d)]{refHr}),
we have a natural isomorphism 
$$\pi_*(\OO_X(tH-\sum_im_iE_i))\cong\OO_{\pr{n}}(t)\otimes\pi_*(\OO_X(-\sum_im_iE_i)),$$
so $\shf{I}_{Z}(t)\cong\pi_*(\OO_X(tH-\sum_im_iE_i))$ follows 
if we show that $\shf{I}_Z\cong\pi_*(\OO_X(-\sum_im_iE_i))$.

This is trivial if $n=1$, since then blowing up has no effect.
So assume $n>1$. For convenience we write $\shf{L}$
for $\OO_X(-\sum_im_iE_i)$, notationally suppressing its dependence  
on the $m_i$. 

We start by noting that $\pi_*\OO_X=\OO_{\pr{n}}$. 
(See the argument of \cite[Corollary III.11.4]{refHr}:
since $\pi$ is projective by \cite[Proposition II.7.16(c)]{refHr},
$\pi_*\OO_X$ is coherent. Thus
$\pi_*\OO_X$ is locally a sheaf of finitely generated $\OO_{\pr{n}}$-modules. 
Since $\pi$ is birational, on any affine open of $\pr{n}$, the ring $B$ given by $\pi_*\OO_X$ and the 
ring $A$ given by
$\OO_{\pr{n}}$ both have the same function field, with $A$ being integrally closed
since $\pr{n}$ is smooth, hence normal and $B$ being module finite over $A$
since $\pi$ is projective and $\OO_X$ and hence $\pi_*\OO_X$ are coherent \cite[Corollary II.5.20]{refHr}; 
i.e., $B$ is an integral extension of the integrally closed
ring $A$, with the same function field, so $A=B$ and thus $\pi_*\OO_X=\OO_{\pr{n}}$.)

Now we show that $\pi_*\OO_X(-m_iE_i)$ is either $\OO_{\pr{n}}$ 
(if $m_i\le0$) or $\shf{I}_{p_i}^{m_i}$ (if $m_i>0$). 
If $m_i\le0$, then we have a morphism $\OO_X\to \OO_X(-m_iE_i)$, hence 
$\OO_{\pr2}=\pi_*\OO_X\to \pi_*\OO_X(-m_iE_i)$. This is clearly an isomorphism except 
possibly at the point $p_i$. Let $p_i\in U$ be an affine open neighborhood.
Consider the commutative diagram
$$
\begin{array}{ccccc}
\OO_{\pr{n}}(U) & \stackrel{\cong}{\to}&\OO_X(\pi^{-1}(U))& \hookrightarrow&\OO_X(-m_iE_i)(\pi^{-1}(U))\\
\downarrow & & & & \downarrow\\
\OO_{\pr{n}}(U \setminus \{p_i\}) & &\to & &\OO_X(-m_iE_i)(\pi^{-1}(U)\setminus E_i)\\
\end{array}
$$
The left vertical arrow is an equality by \cite[Exercise I.3.20]{refHr} (see also \cite[Proposition II.6.3A]{refHr}) or by \cite[Exercise III.3.5]{refHr} and the right vertical arrow is injective since $X$ is integral. The bottom arrow is also an isomorphism (since $U \setminus \{p_i\}\cong \pi^{-1}(U) \setminus\{\pi^{-1}(p_i)\}=\pi^{-1}(U) \setminus E_i$), hence the other arrows are isomorphisms too, whence
$\pi_*\OO_X(-m_iE_i)\cong\OO_{\pr{n}}$. 

If $m_i>0$, consider the canonical morphism 
$\shf{I}_{p_i}^{m_i}\to \pi_*(\pi^{-1}\shf{I}_{p_i}^{m_i}\cdot\OO_X)=\pi_*\OO_X(-m_iE_i)$.
Now, $\pi_*(\pi^{-1}\shf{I}_{p_i}^{m_i})$ is the sheaf associated to the presheaf 
$U\mapsto\shf{I}_{p_i}^{m_i}(U)$, hence $ \pi_*(\pi^{-1}\shf{I}_{p_i}^{m_i}\cdot\OO_X)$ is the sheaf
associated to the presheaf $U\mapsto\shf{I}_{p_i}^{m_i}(U)\cdot\OO_X(\pi^{-1}(U))=
\shf{I}_{p_i}^{m_i}(U)\cdot\OO_{\pr2}(U)=\shf{I}_{p_i}^{m_i}(U)$.
I.e, the canonical sheaf morphism $\shf{I}_{p_i}^{m_i}\to \pi_*\OO_X(-m_iE_i)$
comes from an isomorphism of presheaves, hence is an isomorphism itself.

Thus $\shf{I}_Z\to \pi_*{\shf{L}}$ is locally an isomorphism hence it is an isomorphism, so 
$$H^0(\pr{n},\shf{I}_Z(t))\cong H^0(\pr{n},\pi_*\shf{L}(t))=H^0(X,\OO_X(tH-\sum_im_iE_i)).$$

Now assume that $m_i\ge0$ for all $i$ and let $\shf{L}$ denote $\OO_X(tH-\sum_im_iE_i)$
for an arbitrary integer $t$.
We conclude by applying \cite[Exercise III.8.1]{refHr}, showing that 
$H^l(X,\shf{L})=H^l(\pr{n},\pi_*\shf{L})$ for all $l>0$. For this we need to show
that $R^l\pi_*\shf{L}=0$ for all $l>0$, and to do this it is enough to check that
the stalks vanish. This is clear at points away from each point $p_i$
since $\pi$ is an isomorphism then. Thus $R^l\pi_*\shf{L}$ has support
at most at the points $p_i$, hence at $p_i$ it is equal to the inverse limit
of $H^l(jE_i, \OO_{jE_i}(m_i))$ over $j$ by \cite[Theorem III.11.1]{refHr}
(as in the proof of \cite[Corollary V.3.4]{refHr}), where $\OO_{jE_i}(m_i)$
denotes $\OO_{jE_i}\otimes_{\OO_X}\OO_X(-m_iE_i)$. Thus it suffices to show that
$H^l(jE_i, \OO_{jE_i}(m_i))=0$. Look at the exact sequence
$0\to \OO_X(-jE_i)\to \OO_X\to \OO_{jE_i}\to 0$ and tensor through by
$\OO_X((m_i+j)H-m_iE_i)$ to get 
$0\to \OO_X((m_i+j)(H-E_i))\to \OO_X((m_i+j)H-m_iE_i)\to \OO_{jE_i}(m_i)\to 0$.
The result will follow by showing that $h^l(X,\OO_X(aH-bE_i))=0$ for all $l>0$
if $a\ge b$. 

Let $Y$ be a prime divisor whose class is $H-E_i$ if $b>0$ or just $H$ if $b=0$.
Consider $0\to \OO_X(-Y)\to \OO_X\to \OO_Y\to 0$ and tensor through by $\OO_X(aH-bH_i)$ to get
$0\to \OO_X(a'H-b'E_i)\to \OO_X(aH-bE_i)\to \OO_Y(aH'-bE'_i)\to 0$,
where $H'=H\cap Y$, $E'_i=E_i\cap Y$, $a'=a-1$ and $b'$ is the maximum of $b-1$ and 0. 
Taking cohomology of this exact sequence shows that $h^l=0$ for the ends for all $l>0$
then $h^l=0$ for the middle term for all $l>0$.
Since $\OO_Y(aH'-bE'_i)$ is 
of the same form as what we wish to prove, but in dimension one less, 
and since the result is true in
dimension 1 (i.e., when $Y=\pr1$), we may assume the rightmost term has $h^l=0$ for
all $l>0$ by induction. Showing the same for the leftmost term
eventually reduces to showing $h^l(X, \OO_X)=0$ for all $l>0$. 
For this mimic the argument of \cite[Proposition V.3.4]{refHr}.
\end{proof}

\subsection{Symbolic Powers}

Let $P$ be a prime ideal in a polynomial ring $R=k[x_0,\ldots,x_n]$ over an algebraically closed field $k$.
By the Nullstellensatz, we know that $P=\cap_{P\subseteq M{\ \rm maximal}}M$.
The {\it symbolic power\/} $P^{(m)}$ of $P$ can be defined as 
$P^{(m)}=\cap_{P^m\subseteq M{\ \rm maximal}}M^m$ (see \cite[Theorem 3.14]{refEs}).
This generalizes nicely to the case of an ideal $I(Z)=\cap I(p_i)$ of points $p_1,\ldots,p_s\in\pr{n}$, 
where we define the $m$-th symbolic power $I(Z)^{(m)}$ to be
$$I(Z)^{(m)}=\cap_i(I(p_i)^m).$$ 
I.e., $I(Z)^{(m)}=I(mZ)$, where $mZ$ is the fat point scheme $mp_1+\cdots+mp_s$.
(This is consistent with the definition of symbolic powers used
in \cite{refHH}, in terms of primary decompositions.)

We will for simplicity focus here on the case of symbolic powers of ideals of points in projective space.
See \cite{refPSC} for greater generality.

\subsection{The Containment Problem}

Let $Z=p_1+\cdots+p_s\subset \pr{n}$ and let $I=I(Z)$. Clearly, $I^m\subseteq I^{(m)}$.
In fact, we have: 

\begin{lem}\label{ezcontainment} Let $Z=p_1+\cdots+p_s\subset \pr{n}$ and let $I=I(Z)\subseteq R=k[\pr{n}]$.
Then $I^r\subseteq I^{(m)}$ if and only if $r\ge m$.
\end{lem}

\begin{proof} See Exercise \ref{ex18}.
\end{proof}

Understanding the reverse containment is a much harder largely open problem:

\begin{Prob}\label{ContainmentProb} Let $Z=p_1+\cdots+p_s\subset \pr{n}$ and let $I=I(Z)\subseteq R=k[\pr{n}]$.
Determine all $r$ and $m$ such that $I^{(m)}\subseteq I^r$.
\end{Prob}

Since $I^{(m)}\subseteq I^r$ implies $I^m\subseteq I^{(m)}\subseteq I^r\subseteq I^{(r)}$, 
by Lemma \ref{ezcontainment} we see $m\ge r$. Also, $I^{(1)}=I^1$,
and clearly, $m'\ge m$ implies $I^{(m')}\subseteq I^{(m)}$, so we can restate
Problem \ref{ContainmentProb} as:

\begin{Prob}\label{ContainmentProb2} Let $Z=p_1+\cdots+p_s\subset \pr{n}$ and let $I=I(Z)\subseteq R=k[\pr{n}]$.
Given $r\ge2$, determine the least $m\ge r$ such that $I^{(m)}\subseteq I^r$.
\end{Prob}

As an asymptotic first step, this suggests the following definition and problem:

\begin{Def}[\cite{refBH}]\label{rhodef} Let $Z=p_1+\cdots+p_s\subset \pr{n}$ and let $I=I(Z)\subseteq R=k[\pr{n}]$.
Then define the {\it resurgence} of $I$ to be
$$\rho(I)=\sup\Big\{\frac{m}{r}: I^{(m)}\not\subseteq I^r\Big\}.$$
\end{Def}

\begin{Prob}[\cite{refBH}]\label{rhoProb} Let $Z=p_1+\cdots+p_s\subset \pr{n}$ and let $I=I(Z)\subseteq R=k[\pr{n}]$.
Compute or at least give bounds on $\rho(I)$.
\end{Prob}

It is not clear {\it a priori} that $\rho(I)$ is even finite. Results of Swanson \cite{refSw} showed in many cases that it is
and inspired the results of \cite{refELS} and \cite{refHH}. We state a simplified version of the result of \cite{refHH}:

\begin{thm}\label{HHThm}
Let $Z=p_1+\cdots+p_s\subset \pr{n}$ and let $I=I(Z)\subset R=k[\pr{n}]$.
Then $I^{(nr)}\subseteq I^r$ for each $r\ge 1$, hence $\rho(I)\le n$. 
\end{thm}

Both for \cite{refELS} and for \cite{refHH}, the proof essentially involves finding an ideal
$J$ such that one can check both that $I^{(nr)}\subseteq J$ and that $J\subseteq I^r$; 
\cite{refELS} uses asymptotic multiplier ideals for $J$ (see \cite{refTe}
for an exposition of this approach),
while the proof of \cite{refHH} uses {\it Frobenius} powers
for $J$ (with a dash of tight closure to get the general result). 
Example \ref{HHarg} exhibits the role of Frobenius powers; it is actually a special case of the Hochster-Huneke
proof of Theorem \ref{HHThm}. For the example we will need some results on Frobenius powers:

\begin{Def}\label{Frobpwrdef} Let $I\subseteq R=k[\pr{n}]$ be an ideal.
Assume ${\rm char}(k)=p>0$ and let $q$ be a power of $p$. Define the $q$-th Frobenius power $I^{[q]}$
of $I$ to be the ideal generated by $I^q$.
\end{Def}

\begin{prop}\label{Frobpower} Let $I, J\subseteq R=k[\pr{n}]$ be ideals, where 
${\rm char}(k)=p>0$ and $q$ is a power of $p$.
Then $(I\cap J)^{[q]}=I^{[q]}\cap J^{[q]}$.
\end{prop}

\begin{proof} See \cite[Lemma 13.1.3]{refHS} or \cite[Example 8.4.4]{refPSC}.
\end{proof}

To apply Lemma \ref{Frobpower}, we will also want to note:

\begin{lem}\label{Frobex} Let $I\subseteq R=k[\pr{n}]$ be an ideal generated by
$s$ elements, and assume ${\rm char}(k)=p>0$ and $q$ is a power of $p$.
Then $I^{sq}\subseteq I^{[q]}$.
\end{lem}

\begin{proof} See Exercise \ref{ex19}.
\end{proof}

\begin{Eg}\label{HHarg}
Consider $Z=p_1+\cdots+p_s\subset \pr{n}$ and let $I=I(Z)\subseteq R=k[\pr{n}]$.
Assume ${\rm char}(k)=p>0$ and that $q$ is a power of $p$. Then 
$I(p_i)^{qn}\subseteq I(p_i)^{[q]}$ by Lemma \ref{Frobex} since the ideal of a point in $\pr{n}$
is generated by $n$ linear forms, so
$I^{(qn)}=\cap_i\,I(p_i)^{qn}\subseteq \cap_i\,I(p_i)^{[q]}\subseteq (\cap_i\,I(p_i))^{[q]}\subseteq I^q$
by Proposition \ref{Frobpower} and the obvious fact that $I^{[q]}\subseteq I^q$. 
\end{Eg}

\subsection{Estimating the Resurgence}

In this section we show how to use $\gamma$ from Definition \ref{gammadef} 
and the regularity of an ideal to give bounds on $\rho(I)$.
First we show how to interpret $\gamma$ in this context. Given points $p_1,\ldots,p_s\in\pr{n}$, let
$I=I(Z)$ for $Z=p_1+\cdots+p_s$. Define $\gamma(I)$ to be the infimum of $d_m/m$ 
where $d_m$ is the least degree $t$ such that
$I^{(m)}$ contains a nonzero form of degree $t$. 
(By Proposition \ref{fatptprop}, this is consistent with Definition 
\ref{gammadef}; i.e., $\gamma(I)=\gamma(\pr2;p_1,\ldots,p_s)$.) As in Proposition \ref{gammalimit},
this is actually a limit which is decreasing on multiplicative subsequences;
i.e., $d_{ms}/(ms)\le d_m/m$ for all $s>0$.
More generally, given a homogeneous ideal $0\ne J\subseteq k[\pr{n}]$,
we will denote the least degree $t$ such that $J$ contains a nonzero form of degree $t$
by $\alpha(J)$. Thus $\alpha(J)$ is the degree in which the ideal starts (hence
the use of the first letter, $\alpha$, of the Greek alphabet to denote this concept).
One can also regard $\alpha(J)$ as the $M$-order of $J$, where $M$ is the ideal generated by the variables
(i.e., $\alpha(J)$ is the greatest power of $M$ containing $J$). 

As noted by the remark after Exercise \ref{ex8}, $\varepsilon(\pr2;p_1,\ldots,p_s)\ge1/s$,
hence by Corollary \ref{epsgammacor} we have $\gamma(I)\ge 1$. By a similar argument,
this remains true for $\pr{n}$. In particular, $\gamma(I)>0$,
so it makes sense to divide by $\gamma(I)$.

Given a homogeneous ideal $J\subseteq R=k[\pr{n}]$, for any $t\ge0$ let
$J_t$ be the $k$-vector space span of the forms of degree $t$ in $k$ 
(called the homogeneous component of $J$ of degree $t$).
Note that $R/J$ is also graded; we define $(R/J)_t$ to be $R_t/J_t$.
We recall that the {\it regularity\/} ${\rm reg}(I)$ of $I$ is the least degree $t\ge0$ such that $(R/I)_t$ and $(R/I)_{t-1}$ 
have the same vector space dimension. We have the following theorem:

\begin{thm}[\cite{refBH}]\label{BHthm}
Let $Z=p_1+\cdots+p_s\subset \pr{n}$, let $I=I(Z)$ and let $r$ and $m$ be positive integers. 
\begin{itemize}
\item[(a)] If $\alpha(I^{(m)})<r\alpha(I)$, then $I^{(m)}\not\subseteq I^r$.
\item[(b)] If $r{\rm reg}(I)\le\alpha(I^{(m)})$, then $I^{(m)}\subseteq I^r$.
\item[(c)] $$\frac{\alpha(I)}{\gamma(I)}\le \rho(I)\le \frac{{\rm reg}(I)}{\gamma(I)}$$
\item[(d)] If $\alpha(I)={\rm reg}(I)$, then $I^{(m)}\subseteq I^r$ if and only if $\alpha(I^{(m)})\ge r\alpha(I)$.
\end{itemize}
\end{thm}

\begin{proof}
(a) This is clear, since $\alpha(I^r)=r\alpha(I)$
and so in this case $I^{(m)}$ has a nonzero element of degree less than any nonzero element
of $I^r$.

(b) First we check that $r{\rm reg}(I)\le\alpha(I^{(m)})$ implies that $r\le m$.
Since $I^m\subseteq I^{(m)}$, we see that $\alpha(I^{(m)})\le \alpha(I^m)=m\alpha(I)$.
But $\alpha(I)\le{\rm reg}(I)$ since for all $0\le t< \alpha(I)$ we have 
$\dim_k(R/I)_t>\dim_k(R/I)_{t-1}$. Thus $r{\rm reg}(I)\le\alpha(I^{(m)})\le m\alpha(I)\le m{\rm reg}(I)$.
But for any nonzero ideal $I$ properly contained in $(x_0,\ldots,x_n)$ we have $\alpha(I)>0$, 
so ${\rm reg}(I)>0$ and we see $m\ge r$
and hence $I^{(m)}\subseteq I^{(r)}$.

Now we use the facts that ${\rm reg}(I^r)\le r\,{\rm reg}(I)$ and 
$I^{r}_t=I^{(r)}_t$ for all $t\ge {\rm reg}(I^r)$ \cite{refGGP}; see also \cite{refAV}. 
Thus for $t<r{\rm reg}(I)\le\alpha(I^{(m)})$ we have $0=I^{(m)}_t\subseteq I^r_t$, while for
$t\ge r{\rm reg}(I)\ge {\rm reg}(I^r)$, we have $I^{(m)}_t\subseteq I^{(r)}_t=I^r_t$,
so $I^{(m)}_t\subseteq I^r_t$ holds for all $t$ and we have $I^{(m)}\subseteq I^r$.

(c) For any $0<m/r<\frac{\alpha(I)}{\gamma(I)}$, 
since $\frac{\alpha(I)}{\gamma(I)}=\lim_{s\to\infty}ms\alpha(I)/\alpha(I^{(ms)})$ by Proposition \ref{gammalimit},
for $s\gg0$ we have 
$m/r<ms\alpha(I)/\alpha(I^{(ms)})$, and hence $\alpha(I^{(ms)})< rs\alpha(I)$,
so $I^{(ms)}\not\subseteq I^{rs}$ for $s\gg0$ by (a), hence
$m/r=ms/(rs)\le \rho(I)$; i.e., $\frac{\alpha(I)}{\gamma(I)}\le \rho(I)$.
And for any $m/r\ge{\rm reg}(I)/\gamma(I)$, we have $r{\rm reg}(I)\le m\gamma(I)\le \alpha(I^{(m)})$
so $I^{(m)}\subseteq I^r$ by (b) and hence $ \rho(I)\le \frac{{\rm reg}(I)}{\gamma(I)}$.

(d) If $\alpha(I^{(m)})< r\alpha(I)$, then containment fails by (a), while if $\alpha(I^{(m)})\ge r\alpha(I)$,
then containment holds by (b). 
\end{proof}

\begin{Eg}\label{spts}
Let $I$ be the ideal of $p_1,\ldots,p_s\in\pr{n}$ for $s=\binom{d+n-1}{n}$ general points. 
Then $\alpha(I)={\rm reg}(I)=d$, hence $I^{(m)}\subseteq I^r$ if and only if
$\alpha(I^{(m)})\ge r\alpha(I)$. Unfortunately,
$\alpha(I^{(m)})$ is not in general known. See
however Exercise \ref{ex20}.
\end{Eg}

\subsection{A Question and a Conjecture}

The paper \cite{refBH} gives examples of reduced schemes $Z_i\subset \pr{n}$
of finite sets of points such that $\lim_i\rho(I(Z_i))=n$.
This shows that the bound given in Theorem \ref{HHThm} is in some sense sharp.
However, one can hope to do better. In fact, Huneke has raised the following question:

\begin{Ques}[Huneke]\label{HunekeQues}
Let $I\subset k[\pr{2}]$ be the ideal $I=I(Z)$ where $Z=p_1+\cdots+p_s\subset \pr2$
for a finite set of distinct points $p_i$.
Must it be true that $I^{(3)}\subseteq I^2$?
\end{Ques}

In the case of the ideal $I$ of any $s$ generic points of $\pr2$, \cite{refBH} showed that the answer is yes.
This and additional examples, both in dimension 2 and in higher dimensions, suggested the following conjecture
(this is a simplified version of \cite[Conjecture 8.4.2]{refPSC}):

\begin{Conj}[Harbourne]\label{HarbConj} 
Let $I\subset k[\pr{n}]$ be the ideal $I=I(Z)$ where $Z=p_1+\cdots+p_s\subset \pr{n}$
for a finite set of distinct points $p_i$.
Then $I^{(m)}\subseteq I^r$ if $m\ge rn-(n-1)$.
\end{Conj}

\begin{Eg}\label{charp}
Let $I$ be the ideal of distinct points $p_1,\ldots,p_s\in\pr{n}$. Mimicking the argument of Example
\ref{HHarg} shows in fact that $I^{(rn-(n-1))}\subseteq I^r$ holds if ${\rm char}(k)=p>0$
and $r$ is a power of $p$. See Exercise \ref{ex21}.
\end{Eg}

We thus obtain an observation of Huneke:

\begin{cor}
Question \ref{HunekeQues} has an affirmative answer 
when ${\rm char}(k)=2$.
\end{cor}

Taking $r=2$ in Conjecture \ref{HarbConj} suggests in light of Theorem \ref{BHthm}(a) the following possibly easier question:

\begin{Ques}\label{Ques2}
Let $0\ne I\subset k[\pr{n}]$ be any homogeneous ideal.
Must it be true that $\alpha(I^{(n+1)})\ge 2\alpha(I)$?
\end{Ques}

What is known is that $\alpha(I^{(n+1)})\ge \frac{n+1}{n}\alpha(I)$. (For example, 
if $I$ is the ideal of a set of points,
this follows from \cite[Theorem 1]{refCh}; alternatively, we have $I^{(rn)}\subseteq I^r$
by Theorem \ref{HHThm}, and hence $\alpha(I^{(rn)})\ge r\alpha(I)$, or $\alpha(I^{(rn)})/(rn)\ge \alpha(I)/n$. 
But Proposition \ref{gammalimit} and its proof holds also for $\pr{n}$.
Taking the limit as $r\to \infty$ gives $\gamma(I)\ge \alpha(I)/n$ and, since $\alpha(I^{(m)})/m\ge\gamma(I)$ 
for every $m\ge0$ as in Proposition \ref{gammalimit},
we have $\alpha(I^{(n+1)})\ge \frac{n+1}{n}\alpha(I)$.)  
In fact, examples suggest that $\alpha(I^{(rn-n+1)})\ge r\alpha(I)+n-1$ may hold for the ideal
of any finite set of points in $\pr{n}$ (and perhaps for any nontrivial homogeneous ideal
in $k[\pr{n}]$).

\newpage

\subsection{Exercises}

\begin{Ex}\label{ex18} Prove Lemma \ref{ezcontainment}: 
Let $Z=p_1+\cdots+p_r\subset \pr{n}$ and let $I=I(Z)\subseteq R=k[\pr{n}]$.
Then $I^r\subseteq I^{(m)}$ if and only if $r\ge m$.
\end{Ex}

\begin{Soln}[Exercise \ref{ex18}] If $r\ge m$, then $I^r\subseteq I^m\subseteq I^{(m)}$.
Conversely, assume $I^r\subseteq I^{(m)}$. Localize at $p_1$ and contract to $R_i$, where 
$p_1\in U_i={\rm Spec}(R_i)\cong {\bf A}^n$ is a standard affine open neighborhood
(the complement of $x_i=0$, where $R_i=k[x_0/x_i,\cdots, x_n/x_i]$ for some $x_i$
not vanishing at $p_1$) to get $J(p_1)^r\subseteq J(p_1)^m$, and hence $r\ge m$,
where $J(p_1)\subset R_i$ is the ideal of $p_1$ in $R_i$. 
\end{Soln}

\begin{Ex}\label{ex19} Prove Lemma \ref{Frobex}: 
Let $I\subseteq R=k[\pr{n}]$ be an ideal generated by
$s$ elements, and assume ${\rm char}(k)=p>0$ and $q$ is a power of $p$.
Then $I^{sq}\subseteq I^{[q]}$.
\end{Ex}

\begin{Soln}[Exercise \ref{ex19}]
Let $f_1,\ldots,f_s$ generate $I$. Then monomials in the $f_i$ of degree $sq$
generate $I^{sq}$, and for each such monomial there must be an $i$ such that
$f_i^q$ is a factor. Thus each monomial is in $I^{[q]}$.
\end{Soln}

\begin{Ex}\label{ex20}
Let $I\subset R=k[\pr2]$ be the ideal of non-collinear points 
$p_1,p_2,p_3\in\pr2$. Then $I^{(m)}\subseteq I^r$
if and only if $m\ge \frac{4r-1}{3}$. 
\end{Ex}

\begin{Soln}[Exercise \ref{ex20}] Let $X$ be the blow up of $\pr2$ at the three points.
Note that $B=3L-2(E_1+E_2+E_3)=(L-E_1-E_2)+(L-E_1-E_3)+(L-E_2-E_3)\in\EFF(X)$ 
and $C=2L-E_1-E_2-E_3\in\NEF(X)$ (since $C$ is the class of a prime divisor of positive self-intersection).
Now consider $m=2s+i$ for $s\ge0$ and $0\le i\le 1$. Then $\alpha(I^{(m)})=3s+2i$,
since $sB+iC=(3s+2i)L-(2s+i)(E_1+E_2+E_3)\in\EFF(X)$ is clear but
$((3s+2i-1)L-(2s+i)(E_1+E_2+E_3))\cdot (2L-E_1-E_2-E_3)<0$. 
In particular, $\alpha(I)=2$. Since $R_2/I_2$ and $R_1/I_1$ both have dimension 3, 
we also have ${\rm reg}(I)=2$.

Thus $I^{(m)}\subseteq I^r$ holds by Theorem \ref{BHthm}(d) for $m=2s+i$ exactly when $3s+2i\ge 2r$.
First check the case that $m$ is even: containment holds exactly when
$3m/2\ge 2r$ which is equivalent to $(3m+1)/2\ge 2r$, or $m\ge (4r-1)/3$. 
Now say $m$ is odd: containment holds exactly when
$(3m+1)/2=(3(m-1)+4)/2\ge 2r$ which is again equivalent to $m\ge (4r-1)/3$. 
\end{Soln}

\begin{Ex}\label{ex21} Justify Example \ref{charp}:
Let $J$ be the ideal of distinct points $p_1,\ldots,p_s\in\pr{n}$.  Mimicking the argument of Example
\ref{HHarg} shows in fact that $J^{(rn-(n-1))}\subseteq J^r$ holds if ${\rm char}(k)=p>0$
and $r$ is a power of $p$.
\end{Ex}

\begin{Soln}[Exercise \ref{ex21}] 
The same argument as given in Example \ref{HHarg} works except that we need to refine the statement of
Lemma \ref{Frobex} so that $I^{qs-(s-1)}\subseteq I^{[q]}$.
Let $f_1,\ldots,f_s$ generate $I$. Then monomials in the $f_i$ of degree $sq-(s-1)$
generate $I^{sq-(s-1)}$, but for each such monomial there must in fact be an $i$ such that
$f_i^q$ is a factor (if not, the monomial has degree at most $s(q-1)$ in the $f_i$,
which is less than $sq-(s-1)$). Thus each monomial is in $I^{[q]}$.
\end{Soln}

\end{document}